\documentclass[preprint,12pt]{elsarticle}
\usepackage[top=3cm, bottom=3cm, left=2.2cm, right=2.2cm]{geometry}
\usepackage{amssymb}    % 数学符号
\usepackage{amsmath}    % 基础公式环境
\usepackage{amsthm}     % 定理、引理、定义环境
\usepackage{mathtools}  % 复杂公式扩展（分段函数、积分等）
\usepackage{cases}      % 优化分段函数排版
\usepackage{hyperref}   % 可点击交叉引用
\usepackage{multirow}  % 用于单元格跨行合并
\usepackage{array}     % 增强表格兼容性，避免语法错误
\usepackage{makecell}  % 单元格内换行核心包（新增）
\usepackage{ragged2e}  % 左右对齐+自动换行优化（必加载）
\usepackage{adjustbox}   % 核心：自动缩放表格到页面宽度（不拉伸文字）

\usepackage[table]{xcolor}

\newtheorem{theorem}{Theorem}[section]    % 定理：1.1, 2.1...
\newtheorem{lemma}[theorem]{Lemma}        % 引理：与定理共享编号
\newtheorem{definition}[theorem]{Definition}  % 定义：与定理共享编号
\newtheorem{proposition}[theorem]{Proposition} % 命题：与定理共享编号
\newtheorem{remark}[theorem]{Remark} 
\newtheorem{corollary}[theorem]{Corollary}
\newcommand{\R}{{\mathbb R}}

\begin{document}

\begin{frontmatter}

%% 标题（填写你的论文标题）
\title{Steady states and dynamics of a higher dimensional thin film equation}

%% 作者（多个作者用 \and 分隔）
\author{Shen Bian \corref{cor1}}
\cortext[cor1]{Corresponding author. E-mail: bianshen66@163.com} % 通讯作者

%% 作者单位（多个单位用 label 区分，按需添加）
\affiliation{organization={Department of Mathematics, Beijing University of Chemical Technology},
            addressline={No. 15 East North Third Ring Road, Chaoyang District},
            city={Beijing},
            postcode={100029},
            country={China}}

%% 摘要（必填，控制在 150-200 词，概括研究目的、方法、结果）

\begin{abstract}
We study a higher-dimensional thin film equation that incorporates competitive effects between aggregation and repulsion, where 
repulsion is modeled by fourth-order diffusion and aggregation by backward second-order degenerate diffusion, with a degenerate diffusion exponent $m>0$. We first conduct a systematic analysis of the existence and geometric properties of steady-state solutions for all $m>0$, revealing a critical threshold $m^*=(d+2)/(d-2)$ for variational compactness and solution structure. For $0 < m < m^*$, we then prove that, under natural regularity constraints, radially decreasing steady states coincide with both the extremals of the Gagliardo-Nirenberg-Sobolev inequality and the global minimizers of the free energy. Moreover, we establish the uniqueness of such steady states for $m \neq 1 + 2/d$. Furthermore, in the supercritical regime $1 + 2/d < m < m^*$, we identify a sharp threshold given by the $L^{m+1}$ norm of the unique radial steady-state solution, which distinguishes between global existence for initial data below the threshold and finite-time blow-up for initial data above the threshold. The main contribution of this work is to use steady-state solutions as a theoretical pivot to construct a unified analytical framework that connects parameter classification, variational structure and dynamic behavior. This framework elucidates how the regularity barrier prevents infinite energy descent and selects stable equilibrium states, and thus predicts the global evolution of the system, thereby providing a unified variational principle for understanding steady-state selection and dynamic bifurcations in such higher-order degenerate diffusion equations.
\end{abstract}

%%Research highlights
\iffalse
\begin{highlights}
\item Scale-invariant critical curve $1/m_1+1/m_2=(d+2)/d$ for the system's free energy.
\item Radial, non-increasing and non-compactly supported steady states $(U_s,V_s)$ provide sharp dynamical thresholds.
\item Sharp dichotomy: global existence or finite-time blow-up via stationary $L^{m_1}$/$L^{m_2}$ norms.
\item $L^{m_1}$/$L^{m_2}$ norms critically balance nonlinear diffusion and nonlocal aggregation.
\end{highlights}
\fi
%% 关键词（3-5 个，数学方向相关，用 \sep 分隔）
\begin{keyword}
Degenerate diffusion \sep Steady-state solutions \sep Global existence \sep Blow-up criterion \sep Functionals
\end{keyword}
% 35K65 \sep 34K21 \sep 35A01 \sep 35B44 \sep 35B38
\end{frontmatter}

%% 正文开始

\section{Introduction}

This paper investigates the fourth-order nonlinear degenerate parabolic equation 
\begin{align}\label{tfsystem}
\left\{
  \begin{array}{ll}
u_t=-  \nabla \cdot (u \nabla \Delta u) - \nabla \cdot (u \nabla u^m), & x \in \R^d, ~t > 0, \\
u(x,0)=u_0(x) \ge 0, & x \in \R^d,
\end{array}
\right.
\end{align}
where the diffusion exponent $m>0$. This model is governed by a fundamental competition between two nonlinear effects: the fourth-order term enforces stabilization and repulsion, whereas the backward second-order diffusion term drives destabilizing aggregation.

In one spatial dimension, equation \eqref{tfsystem} is a special case of the long-wave unstable generalized thin film equation
\begin{align}
u_t + (u^n u_{xxx})_x + (u (u^m)_x)_x = 0,
\end{align}
where $u(x,t)$ represents the height of the evolving free surface. The exponents $m$ and $n$ correspond to the powers in the destabilizing second-order diffusive term and the stabilizing fourth-order diffusive term respectively. This class of models appears in many physical systems involving fluid interfaces \cite{M98,ODB97}. When $n=1$ and $m=1$, the equation describes a thin jet in a Hele–Shaw cell \cite{Tom1,Tom23,Tom36}. When $n=m=3$, the equation describes fluid droplets hanging from a ceiling \cite{Tom26}. Over the past fifteen years, these models have also been the subject of rigorous and extensive mathematical analysis \cite{BP98,Tom15,Tom54,Tom45}. 

In this paper, we consider equation \eqref{tfsystem} in the multi-dimensional setting with $d \ge 3$. Initial data will be assumed throughout this paper
\begin{align}\label{initial}
u_0 \in L^1(\R^d)\cap L^{m+1}(\R^d),\quad  \nabla u_0 \in L^2(\R^d), \quad \int_{\R^d} |x|^2 u_0(x) dx< \infty. 
\end{align}
Equation \eqref{tfsystem} possesses several a priori estimates. The total mass is conserved
\begin{align}\label{mass}
M:=\int_{\R^d} u_0(x) dx=\int_{\R^d} u(x,t) dx.
\end{align}
It also admits a dissipation principle for the free energy $F(u)$, defined by
\begin{align}\label{Fu}
\left\{
  \begin{array}{ll}
  F(u)=\frac{1}{2}\int_{\R^d} |\nabla u|^2 dx-\frac{1}{m+1} \int_{\R^d} u^{m+1} dx, \\[3mm]
  \frac{d F(u)}{dt}=-\int_{\R^d} u \left|\nabla(\Delta u + u^m)\right|^2 dx \le 0,
  \end{array}
\right.
\end{align}
where the second identity is obtained by multiplying equation \eqref{tfsystem} by $\mu=-\Delta u-u^m$. The main property of this energy is that it consists of two terms: the gradient energy $\frac{1}{2}\int_{\R^d} |\nabla u|^2 dx$ and the potential energy $-\frac{1}{m+1} \int_{\R^d} u^{m+1} dx$. The difference in signs allows for a competition between aggregation and repulsion. 

Note that equation \eqref{tfsystem} possesses a scaling invariance. If $u(x,t)$ is a solution to equation \eqref{tfsystem}, then the scaled function 
\begin{align}\label{ulambda}
u_\lambda(x)=\lambda^{\frac{2}{m-1}} u(\lambda x, \lambda^{\frac{4m-2}{m-1}} t)
\end{align}
is also a solution to \eqref{tfsystem}. This scaling preserves the $L^p$ norm
\begin{align}
\|u_\lambda\|_{L^p(\R^d)}=\|u\|_{L^p(\R^d)},\qquad p:=\frac{d(m-1)}{2}.
\end{align}
For the critical case $m=1+2/d$, the scaling reduces to the mass-invariant scaling $u_\lambda=\lambda^d u(\lambda x,\lambda^{d+4}t)$, under which the second-order aggregative term $\lambda^{dm+d+2} \nabla \cdot (u_{\lambda} \nabla u_{\lambda}^m)$ and the fourth-order repulsive term $\lambda^{2d+4} \nabla \cdot (u_{\lambda} \nabla \Delta u_\lambda)$ balance. For the subcritical case $0<m<1+2/d$, aggregation dominates repulsion for small $\lambda$, which prevents spreading. While for large $\lambda$, repulsion dominates aggregation, thus precluding blow-up. On the contrary, for the supercritical case $m>1+2/d$, aggregation dominates repulsion for large $\lambda$, and finite-time blow-up may occur. For small $\lambda$, repulsion dominates aggregation, allowing infinite-time spreading. This type of competition between aggregation and repulsion is a common feature in many nonlinear models, including Hele-Shaw flow, stellar collapse, chemotaxis models, and the nonlinear Schrödinger equation (see \cite{BL13,BCL09,BCKS99,B02,C03,C67,F01,KS70,M89,P07,SS99,W83}). In these models, the precise balance of the competing mechanisms critically governs the dynamics of solutions, ultimately determining the dichotomy between global existence and finite-time blow-up. 

For the subcritical case $0<m<1+2/d$, this is the repulsion-dominated regime. In one dimension, \cite{BP98} proved that blow-up is impossible for $m<3$. For higher-dimensional cases with $m<1+2/d$, \cite{Jose24} established the global existence of gradient-flow solutions for general initial data. 

For the critical case $m=1+2/d$, this is the fair competition regime where repulsion and aggregation balance. It has been the subject of extensive research in the existing literature, yielding a wealth of significant results in both one dimension and higher dimensions. In one dimension, \cite{WBB04} showed that there exists a critical mass $M_*$ defined by the best constant of the Sz. Nagy inequality \cite{Nagy41}, such that solutions exist globally in time if the initial mass is less than $M_*$, whereas there are blowing-up solutions if the initial mass is larger than $M_*$. Moreover, $M_*$ is the mass of the compactly supported steady states, which have also been studied in detail. Furthermore, in higher dimensions, \cite{LW17} determined a critical mass $M_c$ via the optimisers of a generalised Nagy inequality and proved finite-time blow-up of solutions for $M>M_c$. We point out that establishing the global existence of solutions in higher-dimensional cases is far more challenging than in one-dimensional cases due to stricter regularity estimates, including discussions on compactness and estimates of higher-order regularity. \cite{Jose24} employed the classical variational minimising movement scheme and the flow interchange technique to derive suitable regularity for weak solutions and proved the global existence of solutions for $M<M_c$.

For the supercritical case $m>1+2/d$, this is the aggregation-dominated regime. In one dimension, \cite{BP00} established the existence of finite-time blow-up solutions for $m>3$, and properties of steady states were also investigated in \cite{Tom42,Tom43,Tom44}. For higher-dimensional cases, to the best of our knowledge, there is little information regarding global existence or finite-time blow-up. This is seemingly due to the fact that mass conservation no longer plays a primary role, and the geometric structure of possible singularities is far more complex than that in lower-dimensional cases. The main goal of this paper is to investigate the global existence and finite-time blow-up of solutions in the supercritical case in higher dimensions, as well as the existence of steady states for all $m>0$.

We emphasize that, in contrast to this well-documented mass-critical exponent, there exists a distinct and fundamental energy-critical exponent $m=(d+2)/(d-2)>1+2/d$ within the supercritical regime, which has received little attention to date. In fact, substituting the scaling \eqref{ulambda} into the free energy $F(u)$ yields
\begin{align}
F(u_\lambda)=\lambda^{\frac{2(m+1)}{m-1}-d} F(u).
\end{align}
The free energy is invariant when $m=(d+2)/(d-2)$. For this reason, we refer to $(d+2)/(d-2)$ as the energy-critical exponent. More importantly, it also serves as the critical exponent for the steady-state solutions to \eqref{tfsystem} (see Section \ref{stationary} below). 

In this work, we first investigate the existence and fundamental properties of steady-state solutions for all $m>0$, and prove that no radial steady-state solutions exist when $m > (d+2)/(d-2)$. For the range $0<m<(d+2)/(d-2)$, we further show that the extremal functions of the Gagliardo-Nirenberg-Sobolev (GNS) inequality (see \eqref{GNS} below) under the corresponding constraints are exactly radially decreasing, compactly supported steady-state solutions, which also attain the global minimum of the free energy. Notably, this steady-state solution is unique when $m \neq 1+2/d$. In particular, for $1+2/d < m < (d+2)/(d-2)$, we establish a sharp threshold based on the $L^{m+1}$-norm of the unique steady-state solution: solutions exist globally if the initial data lies below this threshold, while finite-time blow-up occurs if the initial data exceeds it. 

The contribution of this work is threefold: to characterize the variational structure of the free energy, to identify ground-state solutions as extremals of the GNS inequality, and to establish a sharp threshold that separates global existence from finite-time blow-up. Our results also clarify the role of the regularity barrier in stabilizing equilibrium states. These results are summarized in Table \ref{table1} below.

The remainder of this paper is organized as follows. Section \ref{sec1} is devoted to the precise statements of the main results. Section \ref{sec2} establishes the existence and fundamental properties of steady-state solutions. Subsequently, Section \ref{sec3} investigates the variational structure of the GNS inequality and identifies global minimizers of the free energy. Finally, in the supercritical regime, using the unique free-energy minimizer, Section \ref{sec4} establishes a sharp threshold on the initial data that distinguishes between global existence and finite-time blow-up of solutions.

In what follows, we denote by $C$ a generic constant (which may vary from line to line) and by $C=C(\cdot,\cdots,\cdot)$ a constant depending only on the quantities in parentheses. We also use the simplified notation $\|\cdot\|_{r}:=\|\cdot\|_{L^r(\R^d)}$ for $1 \le r <\infty$.

\section{Main results and preliminaries}\label{sec1}

This section is devoted to the statement of our main results, which are organized into three parts. First, Section \ref{stationary} concerns the existence and basic properties of steady-state solutions. Second, Section \ref{GNSmini} addresses the connection between radially symmetric steady states and extremals of the GNS inequality. Third, Section \ref{dynamic} provides a sharp threshold for the dynamical behavior of solutions in the supercritical regime.

Table \ref{table1} provides a unified summary of the exponent regimes for both steady-state and dynamical solutions of \eqref{tfsystem}. The detailed statements follow.

\begin{table}[htb]
  \centering
  \caption{Summaries of main results}
  \label{table1}
   \vspace{.1in}
  \adjustbox{max width=\linewidth}{
    \begin{tabular}{|c|c|c|c|c|}
      \hline
      Regime & $m$ & Dynamic solutions & Steady states & Free energy \\ \hline
      
      % Row 1
      \multirow{3}{*}{ \makecell{ Supercritical case \\ $m>1+2/d$ } } & 
      \makecell[c]{\vspace{3pt}$m>\frac{d+2}{d-2}$\vspace{3pt}} & 
      \multirow{2}{*}{ } & 
      \makecell{All stationary solutions \\ are compactly supported, \\ no radially decreasing \\ solutions (Theorem \ref{steadythm}(iii)). } & 
       \\ 
      \cline{2-4}
      
      % Row 2
      &   \makecell[c]{\vspace{2pt} Energy-critical case: \\ $m=\frac{d+2}{d-2}$ \vspace{0pt}}  & \makecell{ } 
      & \makecell{There are radially \\ symmetric, non-compactly \\ supported stationary \\ solutions (Theorem \ref{steadythm}(ii)).} & 
      \multirow{2}{*}{} \\  % 跨2行
      \cline{2-4} \cline{5-5}
      
      % Row 3
      & \makecell[l]{\vspace{3pt} $1+\frac{2}{d}<m<\frac{d+2}{d-2} $\vspace{0pt}} & \makecell{$L^{m+1}$-norm of a unique \\
      radially decreasing steady \\ state provides a sharp \\ critical threshold  \\ (Theorem \ref{finiteblowup}, Theorem \ref{globalexistence}). } & 
      \multirow{3}{*}{\makecell{All stationary solutions \\ are compactly supported \\ (Theorem \ref{steadythm}(i)). }} & \makecell{ Unique global minimizer\\for any given mass and\\ its corresponding $L^{m+1}$ \\ -norm (Proposition \ref{picture}).}
      \\  % 这里为空
      \cline{1-3} \cline{5-5}
      
      % Row 4
      \makecell{Mass-critical case} & 
      \makecell[l]{\vspace{3pt}$m=1+\frac{2}{d}$\vspace{3pt}} & \makecell{Solutions blow up \\ for large mass\cite{LW17} \\ and exist globally \\ for small mass \cite{Jose24,WBB04}.}
      & & \makecell{Infinitely many \\ global minimizers\\for a specific mass \\ (Proposition \ref{picture1}). } \\
      \cline{1-3} \cline{5-5}
      
      % Row 5
      \makecell{Subcritical case} & 
      \makecell[l]{\vspace{3pt}$0<m<1+\frac{2}{d}$\vspace{3pt}} & 
      \makecell{All solutions exist \\ globally in time \cite{Jose24}} & & \makecell{Unique global minimizer \\ for any given mass \\ (Proposition \ref{picture2})} \\
      \hline
    \end{tabular}
  }
\end{table}

\subsection{Steady-state solutions for all $m>0$} \label{stationary}

Let $U_s$ be a steady-state solution to \eqref{tfsystem} and its support be
\begin{align*}
\Omega := \{x \mid U_s(x) > 0\}.
\end{align*}
Our first result characterizes the steady-state solutions of \eqref{tfsystem}, which are summarized as follows:
\begin{enumerate}
\item[(i)] For all $m>0$, any steady-state solution satisfies the elliptic equation
      \begin{align}\label{ellip}
        -\Delta U_s-U_s^m=\frac{1}{M}\frac{(d-2)m-(d+2)}{(d+2)(m+1)} \|U_s\|_{m+1}^{m+1} \quad \text{a.e. in } \Omega.
      \end{align}
      See Proposition \ref{prop1}.
\item[(ii)] For $0<m<(d+2)/(d-2)$, all nonnegative steady-state solutions are compactly supported. If the support $\Omega$ is a star-shaped domain with $C^1$ boundary, then $\frac{\partial U_s}{\partial \vec{n}}=0$ on $\partial\Omega$. See Theorem \ref{steadythm}(i).    
\item[(iii)] For $m=(d+2)/(d-2)$:
    \begin{enumerate}
        \item[(a)] If the support is the entire space ($\Omega = \R^d$), then $U_s$ is uniquely given by the explicit radial formula
        \begin{align}
           U_s=\left( \frac{\sqrt{d(d-2)} \lambda}{\lambda^2+|x-x_0|^2}\right)^{\frac{d-2}{2}} \quad \text{for some } \lambda>0,\; x_0 \in \R^d.
       \end{align}    
        \item[(b)] If $\Omega$ is a star-shaped bounded domain, no positive steady-state solutions exist.
        \item[(c)] For a bounded but non-star-shaped domain, the existence of positive steady-state solutions depends on the geometry of $\Omega$.
    \end{enumerate}
    See Theorem \ref{steadythm}(ii).
\item[(iv)] For $m>(d+2)/(d-2)$, all nonnegative steady-state solutions are compactly supported, and no positive steady-state solutions exist when $\Omega$ is star-shaped. See Theorem \ref{steadythm}(iii).
\end{enumerate}

These results demonstrate that $m=(d+2)/(d-2)$ marks a critical transition in the support properties of steady-state solutions. For $m \neq (d+2)/(d-2)$, every nonnegative steady-state solution is compactly supported, whereas for $m=(d+2)/(d-2)$, non-compactly supported steady-state solutions exist.

\subsection{GNS extremals and global minimizers of the free energy for $0<m<(d+2)/(d-2)$}\label{GNSmini}

The main tool for the analysis of dynamical solutions to \eqref{tfsystem} is the following Gagliardo-Nirenberg-Sobolev inequality (see Section \ref{sec3}): for $0<m<(d+2)/(d-2)$, there exists an optimal constant $C_\ast$ such that
\begin{align}\label{GNS}
\|u\|_{m+1}^{\alpha+2} \le C_\ast \|u\|_{1}^\alpha \|\nabla u\|_{2}^2,\quad \alpha=\frac{d+2-(d-2)m}{dm}.
\end{align}
Throughout this work, for a given mass $M>0$, we introduce the set
\begin{align}\label{YM}
\mathcal{Y}_M:=\{ u \in L_+^1(\R^d) \cap H^1(\R^d): \|u\|_1=M \}
\end{align}
In the mass-critical case $m=1+2/d$, we define
\begin{align}\label{YMc}
\mathcal{Y}_{M_c}:=\{ u \in L_+^1(\R^d) \cap H^1(\R^d): \|u\|_1=M_c \},
\end{align}
where
\begin{align}\label{Mc}
M_c:=\left( \frac{m+1}{2C_*} \right)^{d/2}.
\end{align}
For $m \neq 1+2/d$, we define 
\begin{align}\label{YMPstar}
\mathcal{Y}_{M,P_\ast}:=\{ u \in \mathcal{Y}_M: \|u\|_{m+1}=P_\ast \},
\end{align}
where
\begin{align}\label{Pstar}
P_\ast:=\left( \frac{\alpha+2}{2 C_\ast M^\alpha} \right)^{\frac{1}{m-\alpha-1}}.
\end{align}
Note that the condition $m \neq 1+2/d$ ensures $m-\alpha-1 \neq 0$, so $P_\ast$ is well-defined.

Our second result characterizes the extremal functions of the GNS inequality \eqref{GNS} and the global minimizers of the free energy $F(u)$.
\begin{itemize}
  \item GNS extremals and steady-state solutions (see Theorem \ref{Ustar})
  \begin{itemize}
  \item[(i)] For $0<m<(d+2)/(d-2)$ and $m \neq 1+2/d$, within $\mathcal{Y}_{M,P_\ast}$, the GNS inequality \eqref{GNS} has a unique optimizer $U_\ast$, which is the unique radially decreasing, compactly supported steady-state solution to \eqref{tfsystem}. 
  \item[(ii)] For $m=1+2/d$, within $\mathcal{Y}_{M_c}$, the GNS inequality \eqref{GNS} has infinitely many optimizers. We denote this family by $u_c$. These optimizers are radially decreasing, compactly supported steady-state solutions to \eqref{tfsystem}, and form a one-parameter family under mass-invariant scaling.
  \end{itemize}
  \item Global minimizers of the free energy (see Section \ref{133})
    \begin{itemize}
     \item[(i)] For $1+2/d<m<(d+2)/(d-2)$, $U_\ast$ is the unique global minimizer of $F(u)$ in $\mathcal{Y}_{M,P_\ast}$, and also the unique global minimizer of $F(u)$ among all steady-state solutions to \eqref{tfsystem} with mass $M$. Moreover, we have 
         \begin{align}
            \displaystyle \inf_{u \in \mathcal{Y}_{M,P_\ast}} F(u)=F(U_\ast)=\frac{dm-(d+2)}{(d+2)(m+1)} P_\ast^{m+1}>0.
         \end{align}   
        See Proposition \ref{picture}.
      \item[(ii)] For $m=1+2/d$, the optimizers $u_c$ of the GNS inequality are also the global minimizers of $F(u)$ in 
      $\mathcal{Y}_{M_c}$ and satisfy $F(u_c)=0$. See Proposition \ref{picture1}.
      \item[(iii)] For $0<m<1+2/d$, $U_\ast$ is the unique global minimizer of $F(u)$ in $\mathcal{Y}_M$. Moreover, we have
             \begin{align}
                 \displaystyle \inf_{u \in \mathcal{Y}_M} F(u)=F(U_\ast)=\frac{dm-(d+2)}{(d+2)(m+1)} P_\ast^{m+1}<0.
             \end{align}
       See Proposition \ref{picture2}.
   \end{itemize}
\end{itemize}

\subsection{Dynamical solutions for $1+2/d<m<(d+2)/(d-2)$}\label{dynamic}

We first state the definition of the weak solution that will be used in Section \ref{sec4}.

\begin{definition}[Weak solution]\label{weakdefine}
Let $u_0$ be the initial data satisfying \eqref{initial} and $T \in (0,\infty]$. A weak solution $u$ to \eqref{tfsystem} is a nonnegative function satisfying
\begin{align*}
& u \in L^\infty \bigl(0,T; H^1(\R^d) \bigr)\cap L^\infty \bigl(0,T; L^1 \cap L^{m+1}(\R^d) \bigr), \\
& u^{1/2} \nabla \Delta u \in L^2\bigl(0,T;L^2(\R^d)\bigr)
\end{align*}
such that for every $0<t<T$ and any $\psi \in C_0^\infty(\R^d)$,
 \begin{align}
 \int_{\R^d} \psi u(\cdot,t)\,dx- \int_{\R^d} \psi u_0(x)\,dx & =\int_0^t
 \int_{\R^d} \nabla \psi \cdot u \nabla \Delta u\,dx\,ds \nonumber \\
 &-\frac{m}{m+1} \int_0^t \int_{\R^d} \Delta \psi\,u^{m+1}\,dx\,ds.  \label{weak}
 \end{align}
\end{definition}

Our third result concerning dynamical solutions in the supercritical regime is presented as follows:
\begin{itemize}
 \item For $1+2/d<m<(d+2)/(d-2)$, we assume $F(u_0)<F(U_\ast)$. 
 \begin{enumerate}
    \item[(i)] If $\|u_0\|_{m+1}>\|U_\ast\|_{m+1}$, then there exist solutions that blow up in finite time. See Theorem \ref{finiteblowup}.
    \item[(ii)] If $\|u_0\|_{L^{m+1}(\R^d)}<\|U_\ast\|_{m+1}$, then the solution exists globally in time, and the second moment satisfies $ \displaystyle \lim_{t \to \infty} \int_{\R^d} |x|^2 u(\cdot,t) dx=+\infty$. See Theorem \ref{globalexistence}.    
 \end{enumerate}
\end{itemize}

We remark that for $F(u_0)<F(U_\ast)$, the case $\|u_0\|_{m+1}=\|U_\ast\|_{m+1}$ cannot occur. See Remark \ref{Fu0FUstar}. For $F(u_0)>F(U_\ast)$, we cannot exclude two possibilities: either the second moment of solutions diverges, or the $L^{m+1}$-norm of solutions is unbounded. If both are bounded, the solution may converge to $U_\ast$. We will comment further on these issues in Section \ref{FU0great}.

Our result provides a sharp criterion for the dichotomy between global existence and finite-time blow-up of dynamical solutions for $1+2/d<m<(d+2)/(d-2)$. The principal difficulty we address stems from the fact that, unlike in the critical case, mass conservation alone fails to give adequate scaling control. Our approach overcomes this by systematically exploiting the variational structure of steady-state solutions and their deep connection with the free energy functional. However, for $m \ge (d+2)/(d-2)$, the lack of embedding compactness prevents the steady-state solutions from serving as a tool for determining the behavior of dynamical solutions, leaving an open problem in that case. While previous studies have mainly focused on the critical exponent $m=1+2/d$, the present work extends the analysis to a broader supercritical range, thereby connecting the well-understood mass-critical theory with the less explored energy-critical regime.

\begin{remark}[Physical interpretation: potential well and threshold dynamics] 
The above result reveals a sharp threshold behavior dictated by the potential-well structure in gradient-flow dynamics:
\begin{itemize}
    \item \textbf{Initial energy below the steady-state energy} $F(u_0) < F(U_\ast)$ guarantees that the system has already ``crossed the ridge''. Its evolution can only proceed \textbf{downhill (energy decreasing)} and cannot climb back over the energy barrier.    
    \item \textbf{The initial $L^{m+1}$ norm} $\|u_0\|_{m+1}$ determines \textbf{which downhill path} the system follows.
    \begin{itemize}
        \item[(a)] If $\|u_0\|_{m+1} < \|U_\ast\|_{m+1}$, the system enters a \textbf{gentle valley (inside the potential well)} where $\|u(\cdot,t)\|_{m+1}$ and $\|\nabla u(\cdot,t)\|_{2}$ remain uniformly bounded. Hence, the free energy is uniformly bounded from below and decreases to a limit $F_\infty<F(U_\ast)$. The second moment necessarily becomes unbounded, implying that the mass escapes to spatial infinity and the solution vanishes locally on every compact set. Whether the profile, after a suitable translation, converges to a steady state with energy lower than $F(U_\ast)$ remains an open question. It may instead approach a singular measure at infinity.
        \item[(b)] If $\|u_0\|_{m+1} > \|U_\ast\|_{m+1}$, the system slides into a \textbf{steep ravine (outside the potential well)}, where attraction dominates, the gradient grows rapidly, and finite-time blow-up (rupture) occurs.
    \end{itemize}
\end{itemize}
This threshold mechanism shows that blow-up depends not only on how low the initial energy is, but more crucially on whether the \textbf{initial concentration}, measured by the $L^{m+1}$ norm, exceeds a critical level. Physically, the norm $\|U_\ast\|_{m+1}$ marks the \textbf{branching point of the downhill path}. Below it, dissipation (surface tension) dominates. Above it, attraction (self-focusing) takes over. The condition $F(u_0) < F(U_\ast)$ locks the system into a ``downhill mode'', while the concentration threshold sharply decides between convergence to a regular equilibrium and finite-time singularity.
\end{remark}

\section{Qualitative properties of the steady profiles} \label{sec2}

This section is devoted to the analysis of steady-state solutions to \eqref{tfsystem}. The steady-state equation corresponding to \eqref{tfsystem} is followed in the sense of distributions: 
\begin{align}\label{steadyUs}
\nabla \cdot (U_s \nabla \Delta U_s)+\nabla \cdot (U_s \nabla U_s^m)=0,\quad x \in \R^d.
\end{align}
Proposition \ref{prop1} recasts the fourth-order steady-state equation as a second-order elliptic equation, which corresponds to a constant chemical potential on the support of the solution. Proposition \ref{prop2} analyzes a priori properties of solutions to this second-order equation, such as boundedness of the support and boundary integral conditions. Theorem \ref{steadythm} synthesizes these properties and uses classical existence results to establish the existence, nonexistence, and support properties of steady-state solutions in different parameter regimes.

To streamline the presentation of steady-state solution properties, we first introduce a collection of general assumptions that will be referenced throughout this section:
\begin{itemize}
\item[$(A1)$:] $0<m<\infty$, $U_s \in H^1(\R^d) \cap L^{m+1}(\R^d)\cap C(\R^d)$ with $u \ge 0$ and $\int_{\R^d} U_s dx=M$. Moreover, $\Omega$ is a connected open set defined by $\Omega:=\{ x \mid U_s>0 \}$ and $U_s=0$ in $\R^d\setminus \Omega$.
\item[$(A2)$:] For $m>(d+2)/(d-2)$, we assume $U_s \in C^2(\Omega)$.
\item[$(A3)$:] If $\Omega=\R^d$, we assume $U_s$ decays sufficiently at infinity.
\item[$(A4)$:] If $\Omega$ is bounded, we assume $\partial \Omega \in C^1$ and $U_s \in C^1(\overline{\Omega})$.
\end{itemize} 
For a bounded domain $\Omega$ satisfying $(A4)$, we denote by $\vec{n} = \vec{n}(x)$ the unit outward normal vector on its boundary $\partial \Omega$. The corresponding outward normal derivative is denoted by $\frac{\partial U_s}{\partial \vec{n}} := \nabla U_s \cdot \vec{n}$. When stating our main results, we will simply cite the relevant assumptions by their labels for brevity. 

We now present four equivalent statements for the steady-state solution and show that the chemical potential is constant inside the support of the steady-state solution. 

\begin{proposition}\label{prop1}
Under assumptions $(A1)-(A3)$, we denote 
\begin{align}
\mu_s=-\Delta U_s-U_s^m.
\end{align}
Then the following four statements are equivalent:
\begin{itemize}
\item[(i)] Equilibrium: $\mu_s \in H_{loc}^1(\Omega), U_s \in L_{loc}^\infty(\Omega)$ and 
   \begin{align}\label{mus}
    \nabla \cdot \left(U_s \nabla \mu_s\right)=0, \quad x \in \R^d
   \end{align}
  in the sense of distributions. 
\item[(ii)] No dissipation: $\int_{\R^d} U_s \left| \nabla \mu_s  \right|^2 dx=0$.
\item[(iii)] $U_s$ is a critical point of $F(u)$ that satisfies the Pohozaev identity
 \begin{align}\label{pohozaev}
\frac{d+2}{2} \int_{\R^d} |\nabla U_s|^2 dx=\frac{dm}{m+1} \int_{\R^d} U_s^{m+1} dx. 
\end{align}
\item[(iv)] The chemical potential satisfies
      \begin{align}
        \mu_s=\overline{C} \quad \text{a.e. in } \Omega,
      \end{align}
 where $\overline{C}=\frac{1}{M}\frac{(d-2)m-(d+2)}{(d+2)(m+1)} \|U_s\|_{L^{m+1}(\R^d)}^{m+1}$.
\end{itemize}
\end{proposition}
\begin{proof}
We adopt the cyclic proof method to verify these four equivalent statements, that is (i)$\Rightarrow$ (ii) $\Rightarrow$ (iv) $\Leftrightarrow$ (iii)$\Rightarrow$ (i).

\textbf{(i)$\Rightarrow$ (ii):} We take a cut-off function $0 \le \psi_1(x) \le 1$ given by
\begin{align*}
\psi_1(x)=\left\{
            \begin{array}{ll}
              1, & |x| \le 1, \\
              0, & |x| \ge 2,
            \end{array}
          \right.
\end{align*}
where $\psi_1(x) \in C_0^\infty(\R^d)$. Define $\psi_R(x):=\psi_1(x/R)$, so $\psi_R \to 1$ as $R \to \infty$, and there exists a constant $C_1$ such that $\left| \nabla \psi_R(x) \right| \le C_1/R$ for $x \in \R^d$. We multiply \eqref{mus} by $\mu_s \psi_R^2$ to obtain
\begin{align*}
0 &=-\int_{\R^d} U_s \nabla \mu_s \cdot \nabla \left( \mu_s \psi_R^2 \right)dx \\
&=-\int_{\R^d} U_s \left| \nabla \mu_s \right|^2 \psi_R^2 dx-2 \int_{\R^d} U_s \mu_s \psi_R \nabla \mu_s \cdot \nabla \psi_R dx \\
& \le -\int_{\R^d} U_s \left| \nabla \mu_s \right|^2 \psi_R^2 dx +\frac{1}{2} \int_{\R^d} U_s \left| \nabla \mu_s \right|^2 \psi_R^2 dx+2 \int_{\R^d} U_s \mu_s^2 \left| \nabla \psi_R \right|^2 dx \\
&=-\frac{1}{2} \int_{\R^d} U_s \left| \nabla \mu_s \right|^2 \psi_R^2 dx+2 \int_{\R^d} U_s \mu_s^2 \left| \nabla \psi_R \right|^2 dx.
\end{align*}
We estimate that as $R \to \infty$, 
\begin{align}
\int_{\R^d} U_s \mu_s^2 \left| \nabla \psi_R \right|^2 dx \to 0.
\end{align}
Thus, we obtain $\int_{\R^d} U_s \left| \nabla \mu_s \right|^2 dx=0$. 

\textbf{(ii)$\Rightarrow $ (iv):} Since $U_s=0$ on $\R^d \setminus \Omega$, we have $\int_{\R^d} U_s \left| \nabla \mu_s \right|^2 dx=\int_{\Omega} U_s \left|\nabla \mu_s  \right|^2 dx=0$. It follows from $U_s>0$ at any point $x_1 \in \Omega$ that $\nabla \mu_s=0$ in a neighborhood of $x_1$ and thus $\mu_s$ is a constant in this neighborhood. Using the connectedness of $\Omega$, one has $\mu_s=\overline{C}$ in $\Omega$. 

\textbf{(iii) $\Leftrightarrow $ (iv):} We first derive the associated Pohozaev identity by applying a standard scaling argument \cite{Bere83,BL83II,W83}. Since $U_s$ is a critical point of the free energy $F(u)$, the rescaled function $U_\lambda(x)=\lambda^d U_s(\lambda x)$ yields a one-parameter family 
\begin{align}
G(\lambda)=F(U_\lambda)=\frac{\lambda^{d+2}}{2} \int_{\R^d} |\nabla U_s|^2 dx-\frac{\lambda^{dm}}{m+1} \int_{\R^d} U_s^{m+1} dx
\end{align}
with a stationary point at $\lambda=1$. That is, 
\begin{align}
\left. \frac{d G}{d \lambda} \right|_{\lambda=1} =\frac{d+2}{2} \int_{\R^d} |\nabla U_s|^2 dx-\frac{dm}{m+1} \int_{\R^d} U_s^{m+1} dx = 0.
\end{align} 
This gives the Pohozaev identity
\begin{align}\label{poho30}
\frac{d+2}{2} \int_{\R^d} |\nabla U_s|^2 dx=\frac{dm}{m+1} \int_{\R^d} U_s^{m+1} dx. 
\end{align}
Armed with this identity, we further derive the Euler-Lagrange equation of $U_s$. For any $\psi \in C_0^\infty(\Omega)$, define
\begin{align*}
\varphi(x)=\psi(x)-\frac{U_s}{M} \int_{\Omega} \psi dx.
\end{align*}
Then $\operatorname{supp}(\varphi) \subseteq \Omega$ and $\int_{\Omega} \varphi dx=0$. Moreover, there exists
\begin{align*}
\varepsilon_0: = \frac{ \displaystyle \min_{y \in  \operatorname{supp}(\varphi)} U_s(y)} { \displaystyle \max_{y \in \operatorname{supp}(\varphi)} \left| \varphi (y) \right|}>0,
\end{align*}
so that $U_s+\varepsilon \varphi \ge 0$ in $\Omega$ for all $0<\varepsilon<\varepsilon_0$. Now, $U_s$ is a critical point of $F(u)$ in the set 
\begin{align}
\mathcal{X}_{M}:=\{ u \in L_+^1(\R^d) \cap H^1(\R^d) \cap L^{m+1}(\R^d):~ \|u\|_1=M \} 
\end{align}
if and only if
\begin{align*}
    \frac{d}{d \varepsilon} \bigg |_{\varepsilon=0} F(U_s+\varepsilon \varphi)=0.
\end{align*}
A direct computation yields
\begin{align*}
\int_{\Omega} \left( \mu_{s}-\frac{1}{M} \int_{\Omega} \mu_{s} U_s dx \right) \psi(x) dx=0, \quad \text{for any }\psi \in C_0^\infty(\Omega).
\end{align*}
Therefore, we obtain
\begin{align}\label{eq4}
-\Delta U_s-U_s^m=\frac{1}{M} \left( \int_{\R^d} \left| \nabla U_s \right|^2 dx -\int_{\R^d} U_s^{m+1}dx \right)= \overline{C} \quad \text{a.e. in } \Omega.
\end{align}
Furthermore, we can infer from \eqref{poho30} that
\begin{align}\label{eq8}
    \overline{C}=\frac{1}{M}\frac{(d-2)m-(d+2)}{(d+2)(m+1)} \|U_s\|_{m+1}^{m+1}.
\end{align}

\textbf{(iv)$ \Rightarrow$ (i):} We divide the proof into two cases: $0<m \le (d+2)/(d-2)$ and $m>(d+2)/(d-2)$. For $0<m \le (d+2)/(d-2)$, we infer from $U_s \in H^1(\R^d)$ that $U_s^m \in L_{loc}^r(\Omega)$ for any $r>d/2$. By Moser iteration \cite{Evans}, we obtain $U_s \in L_{loc}^\infty(\Omega)$. Subsequently, the Calderón-Zygmund theory in \cite[Theorem 9.11]{GT01} ensures that $U_s \in W_{loc}^{2,r}(\Omega)$ for any $r \in (1,\infty)$. It then follows from the Sobolev embedding and the Schauder estimate \cite[Theorem 6.2]{GT01} that $U_s \in C_{loc}^{2,\alpha}(\Omega)$ with $\alpha \in (0,1)$. A bootstrap argument then yields $U_s \in C^\infty(\Omega)$. Thus, $U_s$ is a classical solution to \eqref{eq4} in $\Omega$ with $U_s=0$ on $\partial \Omega$. For $m>(d+2)/(d-2)$, by assumption $(A2)$, the standard Schauder estimate and the bootstrap iteration yield $U_s \in C^\infty(\Omega)$. Consequently, $(i)$ follows from $U_s \nabla \mu_s=0$ in $L^2(\R^d)$. This completes the proof.
\end{proof}

According to the identity \eqref{pohozaev}, we proceed to present the following results.
\begin{corollary}\label{Fus}
The free energy of the steady-state solutions satisfies
\begin{align}
F(U_s)=\frac{dm-(d+2)}{(d+2)(m+1)} \|U_s\|_{m+1}^{m+1} \left\{
                                                         \begin{array}{ll}
                                                           >0, & m>1+2/d, \\
                                                           =0, & m=1+2/d, \\
                                                           <0, & 0<m<1+2/d.
                                                         \end{array}
                                                       \right.
\end{align}
\end{corollary}

By virtue of Proposition \ref{prop1}, we have that $U_s$ satisfies the following self-consistent nonlocal elliptic equation
\begin{align}\label{second}
-\Delta U_s-U_s^m=\overline{C} \quad \text{a.e. in } \Omega,
\end{align}  
where 
\begin{align*}
\overline{C}=\frac{1}{M}\frac{(d-2)m-(d+2)}{(d+2)(m+1)} \|U_s\|_{m+1}^{m+1}.
\end{align*}
Now we state a priori properties of $\Omega$ for different values of $m$. 
\begin{proposition}\label{prop2}
Let $U_s$ be a positive solution to \eqref{second}. Under assumptions $(A1)-(A4)$, the following statements hold:
\begin{itemize}
\item[(i)] If $m \neq \frac{d+2}{d-2}$, then $\Omega$ is bounded. If $m=\frac{d+2}{d-2}$, then $\Omega$ is either bounded or the entire space $\R^d$.
\item[(ii)] When $\Omega$ is bounded, the solution satisfies 
    \begin{align}\label{eq0}
      \int_{\partial \Omega} x \cdot \vec{n} \left| \frac{\partial U_s}{\partial \vec{n}} \right|^2 ds=0.
    \end{align}
\end{itemize}
\end{proposition}
\begin{proof}
{\it\textbf{Step 1}} (Proof of (i)) \quad Using the proof of $(iv)\Rightarrow (i)$ in Proposition \ref{prop1}, we know $U_s \in C^\infty (\Omega)$. When $0<m<(d+2)/(d-2)$, we have $\overline{C}<0$. We prove that $\Omega$ is bounded by contradiction. Suppose $\Omega$ is unbounded. As $|x|\to\infty$, the term $U_s^m \to 0$ (since $U_s \to 0$ and $m>0$). Substituting this asymptotic property into Equation \eqref{second}, we get
\begin{align}
\Delta U_s \to -\overline{C}, \quad \text{as } |x| \to \infty.
\end{align}
This implies 
\begin{align}
U_s(x)\sim K |x|^2 \to +\infty, \quad \text{as } |x| \to \infty
\end{align}
for some positive constant $K>0$. This result directly contradicts $U_s \in L^1(\R^d) \cap L^{m+1}(\R^d)$. Thus, $\Omega$ must be bounded. 

Similarly, for $m>(d+2)/(d-2)$, the same asymptotic contraction argument shows that $\Omega$ is bounded. For $m=(d+2)/(d-2)$, we have $\overline{C}=0$, so equation \eqref{second} becomes homogeneous. Therefore, $\Omega$ can be either bounded or unbounded. This completes the proof of (i).

{\it\textbf{Step 2}} (Proof of (ii)) \quad When $\Omega$ is bounded, we derive a Pohozaev-type identity to obtain a key boundary integral for $U_s$. First, multiplying \eqref{second} by $x \cdot \nabla U_s$, we obtain
\begin{align}\label{eq1}
\frac{2-d}{2} \int_{\Omega} \left| \nabla U_s \right|^2 dx-\frac{1}{2}\int_{\partial \Omega} x \cdot \vec{n} \left| \frac{\partial U_s}{\partial \vec{n}} \right|^2 ds=-\frac{d}{m+1} \int_{\Omega} U_s^{m+1} dx-d ~\overline{C} \int_{\Omega} U_s dx.
\end{align}
Moreover, multiplying \eqref{second} by $U_s$ leads to
\begin{align}\label{eq2}
\int_{\Omega} \left| \nabla U_s  \right|^2 dx=\int_{\Omega} U_s^{m+1} dx+\overline{C} \int_{\Omega} U_s dx.
\end{align}
Combining \eqref{eq1} and \eqref{eq2}, we obtain the Pohozaev-type identity
\begin{align}\label{poho}
\frac{d+2}{2} \int_{\Omega} \left| \nabla U_s \right|^2 dx-\frac{1}{2} \int_{\partial \Omega} x \cdot \vec{n} \left| \frac{\partial U_s}{\partial \vec{n}} \right|^2 ds=\frac{dm}{m+1} \int_{\Omega} U_s^{m+1} dx. 
\end{align}
By virtue of identity \eqref{pohozaev}, we conclude that
\begin{align}\label{eq26}
  \int_{\partial \Omega} x \cdot \vec{n} \left| \frac{\partial U_s}{\partial \vec{n}} \right|^2 ds=0.
\end{align}
This completes the proof of $(ii)$. 
\end{proof}

Equation \eqref{second} is a self-consistent nonlocal elliptic equation. The existence of solutions to \eqref{second} 
follows arguments essentially similar to those for the local second-order elliptic equation in the subcritical case $m<(d+2)/(d-2)$ \cite{A76,ABC94,AR73,Bere83,K57,L84,NT91}, the critical case $m=(d+2)/(d-2)$ \cite{BC88,BN83,CGS89,GNN79,pohozaev3,T76}, 
and the supercritical case $m>(d+2)/(d-2)$ \cite{C84,D88,GT01,KW75,pohozaev3}. Applying the results of Proposition \ref{prop1}, Proposition \ref{prop2}, and known results for \eqref{second}, we summarize the results for the steady-state solutions to \eqref{tfsystem} into one theorem.
\begin{theorem}\label{steadythm}
Under assumptions $(A1)-(A4)$, for any steady-state solution $U_s$ to \eqref{steadyUs}, the following statements hold:
\begin{itemize}
\item[(i)] When $0<m<(d+2)/(d-2)$, every nonnegative solution $U_s$ is compactly supported. If $\Omega$ is a star-shaped bounded domain, then $\frac{\partial U_s}{\partial \vec{n}}=0$ on $\partial \Omega$.
\item[(ii)] When $m=(d+2)/(d-2)$, the positive solution $U_s$ is either compactly supported or non-compactly supported. If $\Omega=\R^d$, then every positive solution $U_s$ uniquely assumes the radially symmetric form up to translation,
   \begin{align}\label{Uscritical}
       U_s=\left( \frac{\sqrt{d(d-2)} \lambda}{\lambda^2+|x-x_0|^2}   \right)^{\frac{d-2}{2}} \quad \text{for some } \lambda>0 \text{ and } x_0 \in \R^d. 
   \end{align}
   If $\Omega$ is a star-shaped bounded domain, then \eqref{steadyUs} admits no positive solutions.
\item[(iii)] When $m>(d+2)/(d-2)$, every nonnegative solution $U_s$ is compactly supported. If $\Omega$ is a star-shaped bounded domain, then \eqref{steadyUs} admits no positive solutions. In particular, no radially decreasing positive solution exists.
\end{itemize}
\end{theorem}
\begin{proof}
We first claim that for $0<m<\infty$, if $\Omega$ is a star-shaped bounded domain, then $x \cdot \vec{n} \ge 0$ on $\partial \Omega$ and \eqref{eq0} implies 
\begin{align}\label{eq28}
\frac{\partial U_s}{\partial \vec{n}}=0 \quad \text{a.e. on } \partial \Omega.
\end{align}
Now we divide the proof into three steps in terms of $m$. 

{\it\textbf{Step 1}} (Proof of (i)) \quad For $0<m<(d+2)/(d-2)$, the existence of solutions to \eqref{second} mainly relies on variational methods, phase-plane analysis, and symmetry principles. Based on the above analysis, equation \eqref{second} is derived from the Euler equation of the energy $F(u)$ under the mass constraint $\int_{\Omega} U_s dx=M$. The constrained minimization framework established by Berestycki and Lions \cite{Bere83, NT91} shows that for any $M>0$, there exists at least one positive solution (minimizer), and the compactness of the Sobolev embedding ensures the convergence of minimizing sequences. Furthermore, using the mountain pass lemma of Ambrosetti and Rabinowitz \cite{AR73}, higher-energy solutions can be constructed under symmetric domains or suitable topological conditions, indicating that multiple solutions may coexist. 

In summary, positive solutions to this equation necessarily exist, but uniqueness holds only in special scenarios such as fixed $L^{m+1}$-norm or radial symmetry \cite{PS98}. In general parameter regimes, multiplicity of solutions is a typical feature of the problem. Finally, if $\Omega$ is a star-shaped bounded domain, applying \eqref{eq28} completes the proof of (i). 

{\it\textbf{Step 2}} (Proof of (ii)) \quad For $m=(d+2)/(d-2)$, if $\Omega=\R^d$, then \eqref{second} becomes
\begin{align}\label{eq27}
-\Delta U_s=U_s^{\frac{d+2}{d-2}},\quad x \in \R^d.
\end{align}
It was proved by Aubin \cite{aubin1} and Talenti \cite{T76} that every positive solution $U_s$ to \eqref{eq27} uniquely assumes the radially symmetric form up to translation
\begin{align}
U_s=\left( \frac{\sqrt{d(d-2)} \lambda}{\lambda^2+|x-x_0|^2}   \right)^{\frac{d-2}{2}} \quad \text{for some } \lambda>0 \text{ and } x_0 \in \R^d.
\end{align}
Moreover, on a bounded smooth domain, the existence of solutions to \eqref{second} depends crucially on the geometry and topology of the domain. If $\Omega$ is star-shaped, Pohozaev \cite{pohozaev3} proved that there are no positive solutions. In fact, for $m=(d+2)/(d-2)$, we have $\overline{C}=0$ in \eqref{second}. Hence, $-\Delta U_s>0$, so $U_s$ is superharmonic in $\Omega$. Moreover, $U_s > 0$ in $\Omega$ and $U_s = 0$ on $\partial\Omega$, so $U_s$ attains its minimum value $0$ on the boundary. The Hopf Lemma (strong maximum principle) for superharmonic functions states: If $U_s \in C^2(\Omega) \cap C^1(\overline{\Omega})$ satisfies $-\Delta U_s > 0$ in $\Omega$ and attains a minimum at a boundary point $x_0 \in \partial\Omega$ where an interior sphere condition holds (true for smooth domains), then the outward normal derivative at $x_0$ satisfies
\begin{align*}
\partial_n U_s(x_0) < 0.
\end{align*}
Applying this to any boundary point, we get $\partial_n U_s <0$ on $\partial\Omega$, which contradicts \eqref{eq28}. Therefore, for $m  =(d+2)/(d-2)$ with $\Omega$ strictly star-shaped, problem \eqref{second} has no positive solutions. This completes the proof of (ii).

{\it\textbf{Step 3}} (Proof of (iii)) \quad For $m>(d+2)/(d-2)$, the existence of solutions to \eqref{second} exhibits a profound dependence on the geometry and topology of the domain. If the domain is star-shaped, we have $\overline{C}>0$ in \eqref{second}. Hence, $-\Delta U_s>0$ and $U_s$ is superharmonic in $\Omega$. As analyzed similarly in Step 2, the Hopf Lemma together with \eqref{eq28}   implies the nonexistence of positive solutions to \eqref{second} \cite{BN83,pohozaev3}. In particular, any radially decreasing positive solution that vanishes on the boundary must have a ball as its support, which is star-shaped, and therefore cannot exist by the nonexistence result on star-shaped domains. This completes the proof of (iii).
\end{proof}

\begin{remark}
When $m=(d+2)/(d-2)$ and $\Omega=\R^d$, equation \eqref{second} admits the explicit positive solution, that is, the Aubin-Talenti bubble solution. Simple computation yields that the $L^1$ norm of this solution diverges, which contradicts it being a critical point of the energy under mass constraint. This reflects the incompatibility between the scale invariance of the critical exponent and the solution with fixed mass. Nevertheless, $\|U_s\|_{m+1}^{m+1}$ is finite and equals $S_d^{d/2}$, where $S_d$ denotes the best Sobolev constant for the embedding $D^{1,2}(\R^d)\hookrightarrow L^{2d/(d-2)}(\R^d)$. In fact, the solution \eqref{Uscritical} is a classical solution to the fourth-order equation \eqref{tfsystem}, corresponding to the extremum of the energy without mass constraint.
\end{remark}

\begin{remark}
For non-star-shaped domains, the situation differs. When $m=(d+2)/(d-2)$, Bahri and Coron \cite{BC88} showed that positive solutions exist when the domain has nontrivial topology (e.g., an annular or multiply connected domain). When $m>(d+2)/(d-2)$, weak solutions may exist under suitable geometric conditions, although compactness and regularity issues arise \cite{BN83,Moroz15}. Thus, the geometry and topology of the domain play a decisive role.
\end{remark}

\section{GNS extremals and global minimizers of the free energy for $0<m<\frac{d+2}{d-2}$}\label{sec3}

Based on the above analysis, Theorem \ref{steadythm}(iii) implies that for $m>(d+2)/(d-2)$, there are no radially symmetric steady-state solutions to \eqref{tfsystem}. In this section, we focus on the case $0<m<(d+2)/(d-2)$. Section \ref{132} establishes the equivalence between radially symmetric steady-state solutions under mass constraints and the extremal functions of the GNS inequality (see \eqref{GNS1} below). Furthermore, in Section \ref{133}, we show that the global minimizer of the free energy (i.e., the ground state solution) is exactly the radially decreasing, compactly supported steady-state solution to \eqref{tfsystem}.

\subsection{Extremal functions of the Gagliardo-Nirenberg-Sobolev inequality}\label{132}

For $0<m<(d+2)/(d-2)$, we first introduce the functional
\begin{align}\label{Ju}
J(u):=\frac{\|u\|_{m+1}^{\alpha+2}}{\|u\|_1^\alpha \|\nabla u\|_2^2}, \quad \alpha=\frac{d+2-(d-2)m}{dm}.
\end{align}
Then we have the following lemma:
\begin{lemma}
Let $0<m<(d+2)/(d-2)$. For $u \in L^1(\R^d) \cap H^1(\R^d)$, then
\begin{align}
C_*:=\displaystyle \sup_{u \neq 0} J(u) <\infty.
\end{align}
\end{lemma}
\begin{proof}
Applying the Sobolev inequality \cite[pp.202]{lieb202} and H\"{o}lder's inequality for $0<m<(d+2)/(d-2)$, we obtain
\begin{align}
\|u\|_{m+1}^{\alpha+2} \le \|u\|_1^\alpha \|u\|_{2d/(d-2)}^2 \le S_d^{-1} \|u\|_1^\alpha \|\nabla u\|_2^2,
\end{align}
where $S_d$ is the best Sobolev constant. Consequently, $C_*$ is finite and bounded from above by $S_d^{-1}$.
\end{proof}

We next turn to the existence of optimizers for the GNS inequality
\begin{align}\label{GNS1}
\|u\|_{m+1}^{\alpha+2} \le C_* \|u\|_1^\alpha \|\nabla u\|_2^2, \quad \alpha=\frac{d+2-(d-2)m}{dm},
\end{align}
which can be proved by similar arguments as in \cite[Lemma 3.3]{BCL09}.

\begin{proposition}\label{Cstar}(Extremals of the GNS inequality) \quad Let $0<m<\frac{d+2}{d-2}$ and $\alpha = \frac{d+2-(d-2)m}{dm}$. Assume $u \in L^1(\R^d) \cap H^1(\R^d)$. Then the equality in \eqref{GNS1} holds if and only if $u=\lambda V(\mu |x-x_0|)$ for any real numbers $\lambda>0, \mu>0$ and some $x_0 \in \R^d$, where the nonnegative function $V$ is a radially symmetric, non-increasing function that solves the Euler-Lagrange equation
\begin{align}\label{VV2}
-\Delta V=\frac{\alpha+2}{2} \frac{\|V\|_{m+1}^{\alpha+1-m}}{\|V\|_1^\alpha C_*} V^m-\frac{\alpha}{2}\frac{\|V\|_{m+1}^{\alpha+2}}{\|V\|_1^{\alpha+1} C_*} \quad \text{a.e. in supp}(V).
\end{align}
\end{proposition}
\begin{proof}
We aim at proving that the equality in \eqref{GNS1} can be achieved. For this purpose, we consider a maximizing sequence $\{u_j\} \in L^1 \cap H^1(\R^d)$ such that
\begin{align}\label{07040704}
\lim_{j \to \infty} J (u_j)=C_*.
\end{align}
In Step 1, we argue that the maximising sequence $\{u_j\}$ can be assumed to be non-negative, radially symmetric and non-increasing for any $j \ge 0$. Subsequently, in Step 2, we establish that the sequence has a limit $V$ such that the equality holds in \eqref{GNS1}. That is,
\begin{align}
\lim_{j \to \infty} u_j=V, \quad J(V)=C_*.
\end{align}
Finally, the Euler-Lagrange equation solved by $V$ is obtained in Step 3.

{\it\textbf{Step 1}} (Radially symmetric, non-increasing assumption) \quad First, we may assume that $u_j$ is a non-negative, radially symmetric, and non-increasing function for any $j \ge 0$. Indeed, the fundamental identity $\|\nabla u_j\|_2^2=\|\nabla |u_j| \|_2^2$ from \cite[Lemma 7.6]{GT01} implies $J(|u_j|)=J(u_j)$, so that $|u_j|$ is also a maximizing sequence. Next, let $u_j^*$ denote the symmetric decreasing rearrangement of $|u_j|$. Then Riesz's rearrangement inequality \cite[pp.87]{lieb202} ensures that 
$\|u_j^*\|_1=\|u_j\|_1$ and $\|u_j^*\|_{m+1}=\|u_j\|_{m+1}$. Moreover, the classical P\'{o}lya-Szeg{\H{o}} inequality \cite{PS51} yields
\begin{align}
\|\nabla u_j^*\|_2^2 \le \|\nabla u_j\|_2^2.
\end{align}
Consequently, we have
\begin{align*}
J(u_j)=J(|u_j|)=\frac{\|u_j\|_{m+1}^{\alpha+2}}{\|u_j\|_1^\alpha \|\nabla u_j \|_2^2} \le \frac{\|u_j^*\|_{m+1}^{\alpha+2}}{\|u_j^*\|_1^\alpha \|\nabla u_j^* \|_2^2}= J(u_j^*).
\end{align*}
This entails that $(u_j^*)_j$ is also a maximizing sequence.

{\it\textbf{Step 2}} (Existence of the supremum) \quad We next turn to the convergence of the maximizing sequence. Indeed, since $u_j$ is radially symmetric and non-increasing, we find that
\begin{align*}
\int_{\R^d} u_j(x) dx = d \alpha_d
\int_{0}^{\infty}
u_j(r) r^{d-1} dr \ge d \alpha_d \int_0^R u_j(r) r^{d-1} dr \ge \alpha_d u_j(R) R^d
\end{align*}
which gives $u_j(R) \le C R^{-d}$. Moreover, $u_j \in H^1(\R^d)$, and the Sobolev inequality yields
\begin{align}
C \|\nabla u_j\|_2 \ge \|u_j\|_{2d/(d-2)} \ge \left(\alpha_d R^d \right)^{\frac{d-2}{2d}} u_j(R).
\end{align}
Hence, we have
\begin{align}
u_j(R) \le G(R) :=C_0 \inf \left\{ R^{-d}, ~R^{-(d-2)/2} \right\},
\quad \text{for any } R>0.
\end{align}
Since $u_j$ is bounded on $(R,\infty)$ for any $R>0$, we deduce from Helly's selection principle \cite[pp.89]{lieb202} that
there exist a subsequence of ${u_j}$ (not relabeled for
convenience) and a non-negative, non-increasing function $V$ such that $u_j \to V$ point-wisely. Moreover, the condition $0<m<(d+2)/(d-2)$ implies 
\begin{align*}
\left\|G(|x|)\right\|_{m+1}^{m+1}=d\alpha_d
\int_{0}^{\infty} G(r)^{m+1} r^{d-1} dr =C_1 \left[
\int_0^1 \frac{1}{r^{(m+1) \frac{d-2}{2}}} r^{d-1} dr +
\int_1^\infty \frac{1}{r^{d(m+1)}} r^{d-1} dr
\right] < \infty.
\end{align*}
Then the conditions
\begin{align*}
\left\{
  \begin{array}{ll}
    u_j \le G(|x|), \\
    \|G\|_{m+1}< \infty,   \\
    u_j \to V \quad \text{point-wisely}
  \end{array}
\right.
\end{align*}
together with the dominated convergence theorem ensure
\begin{align}
\lim_{j \to \infty} \|u_j\|_{m+1}=\|V\|_{m+1}.
\end{align}
On the other hand, the point-wise convergence of $u_j$ and Fatou's lemma yield $\|V\|_1 \le \|u_j\|_1$ and $\| \nabla V\|_2 \le \| \nabla u_j\|_2$. Hence, we arrive at
\begin{align}
C_*=\lim_{j \to \infty} J(u_j)=\lim_{j \to \infty} \frac{\|u_j\|_{m+1}^{\alpha+2}}{\|u_j\|_1^\alpha \|\nabla u_j \|_2^2} \le \frac{\|V\|_{m+1}^{\alpha+2}}{\|V\|_1^\alpha \|\nabla V \|_2^2}=J(V) \le C_*.
\end{align}
Therefore, we conclude that $J(V)=C_*$. 

{\it\textbf{Step 3}} (Euler-Lagrange equation for $V$) \quad We have already shown that $V$ is a radially symmetric and non-increasing function. Let us now derive the Euler-Lagrange equation for $V$. We introduce
$$\Omega=\left\{ x\in \R^d ~|~ V(x)>0 \right\}.$$
For any $\varphi \in C_0^\infty(\Omega)$, there exists $\varepsilon_0>0$ such that $V+\varepsilon \varphi \ge 0$ for $0<\varepsilon<\varepsilon_0$. Then a direct computation of 
\begin{align}
\left. \frac{d}{d\varepsilon} \right|_{\varepsilon=0} J(V+\varepsilon \varphi)=0
\end{align}
leads to
\begin{align}
0=\int_{\R^d} \left( \Delta V+\frac{\alpha+2}{2} \frac{\| \nabla V \|_2^2}{\|V\|_{m+1}^{m+1}}V^m-\frac{\alpha}{2} \frac{\|\nabla V\|_2^2}{\|V\|_1} \right) \varphi(x) dx,\quad \text{for any } \varphi \in C_0^\infty(\Omega).
\end{align}
This gives that $V$ satisfies the following Euler-Lagrange equation
\begin{align}\label{VV}
-\Delta V=\frac{\alpha+2}{2} \frac{\| \nabla V \|_2^2}{\|V\|_{m+1}^{m+1}}V^m-\frac{\alpha}{2} \frac{\|\nabla V\|_2^2}{\|V\|_1} \quad \text{a.e. in supp}(V),
\end{align}
Moreover, using 
\begin{align}
\|V\|_{m+1}^{\alpha+2} = C_* \|V\|_1^\alpha \|\nabla V\|_2^2,
\end{align}
we can rewrite \eqref{VV} as
\begin{align}\label{V5}
-\Delta V=\frac{\alpha+2}{2} \frac{\|V\|_{m+1}^{\alpha+1-m}}{\|V\|_1^\alpha C_*} V^m-\frac{\alpha}{2}\frac{\|V\|_{m+1}^{\alpha+2}}{\|V\|_1^{\alpha+1} C_*} \quad \text{a.e. in supp}(V).
\end{align}
Finally, for any solution $V$ to \eqref{V5}, the re-scaled function 
\begin{align}\label{WW}
W(x)=\lambda V(\mu x), \text{ for arbitrary } \lambda,\mu>0
\end{align}
also solves \eqref{V5} and satisfies
\begin{align}\label{homoge}
J(V)=J(W).
\end{align}
This completes the proof.
\end{proof}

\begin{remark}
Particularly, taking 
\begin{align}
\lambda=\left( \frac{\alpha \|V\|_{m+1}^{m+1}}{(\alpha+2)\|V\|_1} \right)^{1/m}, \quad \mu=\left( \frac{(\alpha+2)\|\nabla V\|_2^2}{2 \|V\|_{m+1}^{m+1}} \lambda^{m-1} \right)^{1/2}
\end{align}
in \eqref{WW}, we obtain
\begin{align}
-\Delta W=W^m-1, \quad x \in B(0,R).
\end{align}
For $m=1$, from Theorem \ref{steadythm}(i), we find that the extremal function satisfies
\begin{align}
\left\{
  \begin{array}{ll}
    W''+\frac{d-1}{r}W'+W=1, \quad 0<r<R, \\
    W(0)<\infty, \quad W'(0)=0, \\
    W(R)=0,\quad W'(R)=0.
  \end{array}
\right.
\end{align}
A direct computation gives 
\begin{align}
W(r)=1-\frac{J_{\frac{d-2}{2}}(r)}{J_{\frac{d-2}{2}}(R)},
\end{align}
where $J_{\frac{d-2}{2}}(r)$ is the spherical Bessel function of the first kind of order $\frac{d-2}{2}$. 

In particular, for $d=3$, 
\begin{align}
W(r)=1-\frac{R}{r} \frac{\sin r}{\sin R}.
\end{align}
Using $W'(R)=0$, we get 
\begin{align}
R \cos R-\sin R=0,\text{ i.e., } \tan R=R.
\end{align}
\end{remark}

Consequently, recalling definitions \eqref{YMc} and \eqref{YMPstar}, we are now ready to deduce the following result concerning the radial steady states and the extremal functions of the GNS inequality \eqref{GNS1}.

\begin{theorem}\label{Ustar}
Let $0<m<\frac{d+2}{d-2}$ and $\alpha = \frac{d+2-(d-2)m}{dm}$. Then the following statements hold: 
\begin{itemize}
  \item[(i)] If $m \neq 1+2/d$, then in $\mathcal{Y}_{M,P_\ast}$, the GNS inequality \eqref{GNS1} admits a unique optimizer $U_\ast$, which is the unique radially decreasing, compactly supported steady-state solution to \eqref{tfsystem}.
  \item[(ii)] If $m = 1+2/d$, then in $\mathcal{Y}_{M_c}$, the GNS inequality \eqref{GNS1} admits infinitely many optimizers. These optimizers are radially decreasing, compactly supported steady-state solutions to \eqref{tfsystem}, and form a one-parameter family under mass-invariant scaling.
\end{itemize}
\end{theorem}
\begin{proof}
{\it\textbf{Step 1}} (Proof of (i)) \quad For $m \neq 1+2/d$, we restate the definition of $P_\ast$ from \eqref{Pstar} as
\begin{align}
P_\ast=\left( \frac{\alpha+2}{2 C_* M^\alpha} \right)^{\frac{1}{m-\alpha-1}}=\left( \frac{(d+2)(m+1)}{2dm C_* M^\alpha} \right)^{\frac{dm}{(dm-d-2)(m+1)}}.
\end{align}
Then, for a given solution $V_0(x)$ to \eqref{VV2}, we set
\begin{align}\label{WW1}
U_\ast(x)=\frac{1}{\lambda_\ast} V_0 \left( \frac{x}{\mu_\ast} \right),
\end{align}
where
\begin{align}
\lambda_\ast=\left( \frac{M}{\|V_0\|_1} \frac{\|V_0\|_{m+1}^{m+1}}{P_\ast^{m+1}} \right)^{1/m},\quad \mu_\ast=\left( \frac{\lambda M}{\|V_0\|_1} \right)^{1/d}.
\end{align}
It follows that $U_\ast$ is also a solution to \eqref{VV2} satisfying
\begin{align}
-\Delta U_\ast=U_\ast^m+\frac{1}{M}\frac{(d-2)m-(d+2)}{(d+2)(m+1)} \|U_\ast\|_{m+1}^{m+1} \quad \text{a.e. in supp}(U_\ast)
\end{align}
with 
\begin{align}\label{MPstar}
\|U_\ast\|_1=M,\quad \|U_\ast\|_{m+1}=P_\ast.
\end{align}
This indicates that $U_\ast$ is a radially symmetric, non-increasing steady-state solution to \eqref{tfsystem}. 

Next, applying Theorem \ref{steadythm}(i), we conclude that for $0<m<(d+2)/(d-2)$, $U_\ast$ solves the free boundary problem
\begin{align}\label{radialequation}
\left\{
  \begin{array}{ll}
   -u''-\frac{d-1}{r}u'=u^m+\frac{1}{M}\frac{(d-2)m-(d+2)}{(d+2)(m+1)} P_\ast^{m+1}, \quad 0<r<R, \\[2mm]
    u'(0)=0,\quad u(R)=u'(R)=0
  \end{array}
\right.
\end{align}
with $0<R<\infty$. It follows from \cite[Theorem 3]{PS98} that there is a unique solution to \eqref{radialequation}. This completes the proof of (i).

{\it\textbf{Step 2}} (Proof of (ii)) \quad For $m=1+2/d$, one has $\alpha+2=m+1$. Then, \eqref{VV2} reduces to
\begin{align}\label{VV3}
-\Delta V=\frac{\alpha+2}{2} \frac{1}{\|V\|_1^\alpha C_*} V^m-\frac{\alpha}{2}\frac{\|V\|_{m+1}^{m+1}}{\|V\|_1^{\alpha+1} C_*} \quad \text{a.e. in supp}(V).
\end{align}
Hence, a solution $V_c$ to \eqref{VV3} is also a steady-state solution to \eqref{tfsystem} if and only if 
\begin{align}
\|V_c\|_1 = M_c= \left( \frac{m+1}{2C_*} \right)^{1/\alpha}.
\end{align}
Thus, $V_c$ satisfies
\begin{align}\label{VV4}
-\Delta V_c=V_c^m-\frac{m-1}{m+1}\frac{\|V_c\|_{m+1}^{m+1}}{M_c} \quad \text{a.e. in supp}(V_c).
\end{align}
Observe that equation \eqref{VV4} is scaling invariant: if $V_c$ is a solution, then $\mu^d V_c(\mu x)$ is also a solution for any $\mu>0$, with the same mass $M_c$. Therefore, the solutions of \eqref{VV4} form a one-parameter family of radial steady states. This completes the proof of (ii).
\end{proof}

\subsection{Energy landscape and minimization properties of $U_\ast$}\label{133}

In this subsection, using the optimizers of the GNS inequality \eqref{GNS1}, we proceed to investigate the global minimizers of the free energy $F(u)$. We first recall the following well-known result concerning the infimum of $F$ (see \cite{Jose24}), and a brief proof is included for completeness. 

\begin{lemma}[Infimum of the free energy] \label{freeinf}
For $m>1+2/d$, we have
\begin{equation}
\inf_{u \in \mathcal{Y}_M} F(u) = -\infty.  % the infimum can't be obtained. not uniformly bounded below.
\end{equation}
\end{lemma}
\begin{proof}
Consider the scaling $u_\lambda(x) = \lambda^d u(\lambda x)$ for $\lambda>0$. A direct computation yields
\begin{align*}
F(u_\lambda) = \frac{\lambda^{d+2}}{2} \|\nabla u\|_2^2 - \frac{\lambda^{dm}}{m+1} \|u\|
_{m+1}^{m+1}.
\end{align*}
Since $m > 1+2/d$ implies $dm>d+2$, the negative term dominates as $\lambda \to \infty$, so $F(u_\lambda) \to -\infty$. Thus, the infimum of $F$ is $-\infty$.
\end{proof}

By Lemma \ref{freeinf}, the free energy $F(u)$ is unbounded from below in the full space $\mathcal{Y}_M$. In addition, Proposition \ref{prop1} indicates that any steady-state solution to \eqref{tfsystem} is a critical point of $F(u)$. Thus, steady-state solutions cannot be global minimizers of $F(u)$ in this space. However, under suitable additional constraints, such steady states become global minimizers of the free energy. To state this precisely, we first define the set of steady-state solutions with mass $M$ by
\begin{align}
\mathcal{S}_M := \{ u \in \mathcal{Y}_M : u \text{ is any steady-state solution to } \eqref{tfsystem} \}.
\end{align}

\begin{proposition}[Aggregation-dominated regime]\label{picture}
Let $1+\frac{2}{d}<m<\frac{d+2}{d-2}, \alpha=\frac{d+2-(d-2)m}{dm}$, and $P_\ast$ be defined by \eqref{Pstar}. Let $U_\ast$ be the unique optimizer from Theorem \ref{Ustar}(i). Then the following statements hold:
\begin{itemize}
   \item[(i)] $U_\ast$ is the unique global minimizer of $F(u)$ in $\mathcal{Y}_{M,P_\ast}$.
   \item[(ii)] $U_\ast$ is the unique global minimizer of $F(u)$ in $\mathcal{S}_M$.
\end{itemize} 
Moreover, 
\begin{align}
\displaystyle \inf_{u \in \mathcal{Y}_{M,P_\ast}} F(u) = \displaystyle \inf_{u \in \mathcal{S}_{M}} F(u)= F(U_\ast)=\frac{dm-(d+2)}{(d+2)(m+1)} P_\ast^{m+1}.
\end{align}   
\end{proposition} 
\begin{proof}
We first state two results directly derived from Theorem \ref{Ustar} and Proposition \ref{prop1}(iii):
\begin{align}
  u \text{ is an optimizer of the GNS inequality \eqref{GNS1} } &\Rightarrow \|u\|_{m+1}^{\alpha+2} = C_* \|u\|_1^\alpha \|\nabla u\|_2^2, \label{mpstar} \\
  u \text{ is a steady-state solution to \eqref{tfsystem} } &\Rightarrow \frac{dm}{m+1} \|u\|_{m+1}^{m+1}=\frac{d+2}{2} \|\nabla u\|_2^2.  \label{sm}
\end{align}

{\it\textbf{Step 1}} (Proof of (i)) \quad For a given mass $M>0$ and any $v \in \mathcal{Y}_{M,P_\ast}$, applying the GNS inequality \eqref{GNS1}, we get
\begin{align}
F(v) & = \frac{1}{2}\|\nabla v\|_2^2-\frac{1}{m+1}\|v\|_{m+1}^{m+1} \nonumber \\
& \ge \frac{1}{2C_\ast M^\alpha} \|v\|_{m+1}^{\alpha+2}-\frac{1}{m+1}\|v\|_{m+1}^{m+1} \nonumber \\
& = \frac{1}{2C_\ast M^\alpha} P_\ast^{\alpha+2}-\frac{1}{m+1}P_\ast^{m+1} \nonumber \\
& = \frac{dm-(d+2)}{(d+2)(m+1)} \|U_\ast\|_{m+1}^{m+1} = F(U_\ast),  \label{Fmin}
\end{align}
where we have used Corollary \ref{Fus}. Thus, for any $v \in \mathcal{Y}_{M,P_\ast}$, it holds that $F(v) \ge F(U_\ast)$. Equality in \eqref{Fmin} is attained if and only if $v=U_\ast$, by virtue of \eqref{mpstar}. This completes the proof of (i).

{\it\textbf{Step 2}} (Proof of (ii)) \quad For any $w \in \mathcal{S}_M$, by the GNS inequality \eqref{GNS1} and identity \eqref{sm}, we obtain
\begin{align}
\|w\|_{m+1}^{\alpha+2} & \le C_* M^\alpha \|\nabla w\|_2^2 \nonumber \\
& = C_* M^\alpha \frac{2dm}{(d+2)(m+1)} \|w\|_{m+1}^{m+1} \nonumber \\
&= \|w\|_{m+1}^{\alpha+2} \frac{\|w\|_{m+1}^{m-\alpha-1}}{P_\ast^{m-\alpha-1}}  = \|w\|_{m+1}^{\alpha+2} \frac{\|w\|_{m+1}^{m-\alpha-1}}{\|U_\ast\|_{m+1}^{m-\alpha-1}}. \label{m1min}
\end{align}
Since $m-\alpha-1>0$ for $m>1+2/d$, we conclude that $\|w\|_{m+1} \ge \|U_\ast\|_{m+1}$ for all $w \in \mathcal{S}_M$. Then Corollary \ref{Fus} assures that
\begin{align}
F(w) =\frac{dm-(d+2)}{(d+2)(m+1)} \|w\|_{m+1}^{m+1}  \ge \frac{dm-(d+2)}{(d+2)(m+1)} \|U_\ast\|_{m+1}^{m+1}=F(U_\ast).
\end{align}
Equality in \eqref{m1min} holds if and only if $w=U_\ast$, thanks to \eqref{mpstar}. This completes the proof of (ii).
\end{proof}

\begin{remark}
Although $F(u)$ is unbounded from below on $\mathcal{Y}_M$ (Lemma \ref{freeinf}), $U_\ast$ is radially decreasing and the unique global minimizer of $F(u)$ in both $\mathcal{S}_M$ and $\mathcal{Y}_{M,P_\ast}$ that satisfies $F(U_\ast)>0$. This makes $ U_\ast$ a saddle point in the full energy landscape: it attracts radial perturbations but is unstable under certain non-radial concentrations.
The boundedness of the gradient and \( L^{m+1} \)-norm acts as a \textit{regularity barrier} that prevents the system from following energy-collapse paths,
thereby selecting the steady state $U_\ast$ as an accessible and stable metastable configuration.
\end{remark}

\begin{proposition}[Fair competition regime]\label{picture1}
Let $m=1+\frac{2}{d}, \alpha=\frac{d+2-(d-2)m}{dm}$, and $M_c$ be defined by \eqref{Mc}. Then there exist infinitely many global minimizers of $F(u)$ in $\mathcal{Y}_{M_c}$, and these minimizers are radially decreasing, compactly supported steady-state solutions $u_c$ to \eqref{tfsystem}. 
Moreover, 
\begin{align}
\displaystyle \inf_{u \in \mathcal{Y}_{M_c}} F(u)=F(u_c)=0.
\end{align}
\end{proposition} 
\begin{proof}
For any $u_c \in \mathcal{Y}_{M_c}$, the GNS equality \eqref{GNS1} gives
\begin{align}
F(u_c) & =\frac{1}{2}\|\nabla u_c\|_2^2-\frac{1}{m+1}\|u_c\|_{m+1}^{m+1}  \nonumber \\
& \ge \frac{1}{2} \frac{\|u_c\|_{m+1}^{\alpha+2}}{C_\ast \|u_c\|_1^\alpha}-\frac{1}{m+1}\|u_c\|_{m+1}^{m+1} \nonumber\\
& = \frac{1}{2} \frac{\|u_c\|_{m+1}^{\alpha+2}}{C_\ast M_c^\alpha}-\frac{1}{m+1}\|u_c\|_{m+1}^{m+1} = 0, \label{Fmass}
\end{align}
since $\alpha+2=m+1$ due to $m=1+2/d$. Moreover, Equality in \eqref{Fmass} holds if and only if $u_c$ is a radially decreasing,  compactly supported steady-state solution to \eqref{tfsystem} with mass $M_c$. This completes the proof. 
\end{proof}
\begin{remark}
The precise value of the critical mass $M_c$ and the existence of a global minimizer at this threshold were established in \cite{Jose24}. We provide an alternative derivation of $M_c$ and further characterize the full set of such minimizers, showing that it coincides with the family of radially decreasing, compactly supported steady states. An analogous minimizer structure appears in the critical regime of the Keller-Segel model studied in \cite[Proposition 5.15]{CGH20}.
\end{remark}

\begin{proposition}[Diffusion-dominated regime]\label{picture2}
Let $0<m<1+\frac{2}{d}, \alpha=\frac{d+2-(d-2)m}{dm}$, and $U_\ast$ be the unique optimizer from Theorem \ref{Ustar}(i). Then $U_\ast$ is the unique global minimizer of $F(u)$ in $\mathcal{Y}_M$. Moreover, 
\begin{align}
\displaystyle \inf_{u \in \mathcal{Y}_M} F(u)=F(U_\ast)=\frac{dm-(d+2)}{(d+2)(m+1)} P_\ast^{m+1}.
\end{align}
\end{proposition} 
\begin{proof}
For any $u \in \mathcal{Y}_M$, applying the GNS inequality \eqref{GNS1} yields 
\begin{align}
F(u) & =\frac{1}{2}\|\nabla u\|_2^2-\frac{1}{m+1}\|u\|_{m+1}^{m+1}  \nonumber \\
& \ge \frac{1}{2} \frac{\|u\|_{m+1}^{\alpha+2}}{C_\ast M^\alpha}-\frac{1}{m+1}\|u\|_{m+1}^{m+1}. \label{sub}
\end{align}
Now define the auxiliary function
\begin{align}
g(x)=\frac{1}{2 C_\ast M^\alpha} x^{\alpha+2}-\frac{1}{m+1}x^{m+1}.
\end{align}
Since $0<m<1+2/d$, we have $\alpha+2>m+1$, so $g(x)$ attains its minimum at $x_*=P_\ast$. Therefore, it follows from \eqref{sub} that
\begin{align}\label{Fge}
F(u)=g\left( \|u\|_{m+1} \right) \ge \frac{1}{2 C_\ast M^\alpha} P_\ast^{\alpha+2}-\frac{1}{m+1} P_\ast^{m+1}=\frac{dm-(d+2)}{(d+2)(m+1)} \|U_\ast\|_{m+1}^{m+1}=F(U_\ast).
\end{align}
Moreover, by Theorem \ref{Ustar}(i), Equality in \eqref{Fge} holds if and only if $u=U_\ast$. This completes the proof.
\end{proof}

\begin{remark}
Interestingly, the above results on the minimizers of the free energy bear similarities to the corresponding results established for the Keller-Segel equation in the critical and subcritical regimes \cite{CCH17,CCH21,CCV15,CHM18,CGH20}.
\end{remark}

\section{Global existence and finite-time blow-up} \label{sec4}

In this section, we use the characterization of $U_\ast$ to investigate dynamical solutions to \eqref{tfsystem} in the supercritical regime $1+2/d<m<(d+2)/(d-2)$. We establish a sharp threshold determined by the $L^{m+1}$ norm of $U_\ast$, such that 
solutions blow up in finite time if the initial data lies above this threshold (Section \ref{sub42}), whereas solutions exist globally if the initial data lies below the threshold (Section \ref{sub41}).

\subsection{A priori estimates for $F(u_0)<F(U_\ast)$}

We first derive a priori estimates under the energy condition $F(u_0)<F(U_\ast)$.

\begin{proposition}\label{belowabove}
Let $1+\frac{2}{d}<m < \frac{d+2}{d-2}, \alpha=\frac{(d+2)-(d-2)m}{dm}$, and $U_\ast$ be the unique optimizer from Theorem \ref{Ustar}(i). Assume that $u(x,t)$ is a weak solution to \eqref{tfsystem} with an initial condition $u_0$ satisfying \eqref{initial}, \eqref{mass} and
\begin{align}\label{Fu0}
 F(u_0)<F(U_\ast).
\end{align}
\begin{enumerate}
  \item[(i)] If
  \begin{align}\label{xiaoyum1}
  \|u_0\|_{m+1}< \|U_\ast\|_{m+1}, 
  \end{align}
  then there exists a constant $\mu_1<1$ such that the corresponding weak solution $u$ satisfies
     \begin{align}\label{xiaoyum}
       \|u(\cdot,t)\|_{m+1} < \mu_1 \|U_\ast\|_{m+1}
    \end{align}
for all $t>0$.
  \item[(ii)] If
  \begin{align}\label{dayum1}
  \|u_0\|_{m+1} > \|U_\ast\|_{m+1},
  \end{align}
  then there exists a constant $\mu_2>1$ such that the corresponding weak solution $u$ satisfies
\begin{align}\label{dayum}
     \|u(\cdot,t)\|_{m+1} > \mu_2 \|U_\ast\|_{m+1}
\end{align}
 for all $t>0$.
\end{enumerate}
\end{proposition}
\begin{proof}
First, applying the GNS inequality \eqref{GNS1}, we infer from the expression of $F(u)$ that
\begin{align}\label{06081}
F(u) & = \frac{1}{2} \|\nabla u\|_2^2 -\frac{1}{m+1} \|u\|_{m+1}^{m+1} \nonumber \\
  & \ge \frac{1}{2C_* M^\alpha} \|u\|_{m+1}^{\alpha+2}-\frac{1}{m+1} \|u\|_{m+1}^{m+1}.
\end{align}
Moreover, substituting the optimizer $U_\ast$ into $F(u)$ yields
\begin{align}\label{06082}
F(U_\ast)& =\frac{1}{2} \|\nabla U_\ast\|_2^2-\frac{1}{m+1} \|U_\ast\|_{m+1}^{m+1} \nonumber \\
& = \frac{1}{2C_* M^\alpha} \|U_\ast\|_{m+1}^{\alpha+2}-\frac{1}{m+1} \|U_\ast\|_{m+1}^{m+1}.
\end{align}
We define the auxiliary function
\begin{align}\label{gx}
g(x)=\frac{1}{2C_* M^\alpha} x^{\alpha+2}-\frac{1}{m+1} x^{m+1}.
\end{align}
By combining \eqref{06081} and \eqref{06082}, we get
\begin{align}
g \left( \|U_\ast\|_{m+1} \right) =F(U_\ast) \ge g \left( \|u\|_{m+1} \right)
\end{align}
for all $u \in L^1 \cap L^{m+1}(\R^d)$. Now, under the condition
\begin{align}\label{414}
F\left(u_0 \right)< F\left(U_\ast \right),
\end{align}
there exists a constant $0<\delta<1$ such that
\begin{align}
F\left(u_0 \right)<\delta F\left(U_\ast \right).
\end{align}
Since $F(u)$ is non-increasing in time, it follows that
\begin{align}\label{0608starstar}
g\left( \|u \|_{m+1} \right)\le F(u) \le F(u_0) < \delta F\left(U_\ast \right)=\delta g \left( \|U_\ast\|_{m+1} \right).
\end{align}

On the other hand, since $\alpha+2<m+1$ for $m>1+2/d$, a direct computation shows that the maximum of $g(x)$ is attained at 
\begin{align*}
x_*=\left( \frac{\alpha+2}{2C_* M^\alpha} \right)^{\frac{1}{m-\alpha-1}}=\|U_\ast\|_{m+1}.
\end{align*}
This implies that $g(x)$ is strictly increasing for $x<\|U_\ast\|_{m+1}$. Therefore, if
\begin{align}
\|u_0\|_{m+1}< \|U_\ast\|_{m+1},
\end{align}
then \eqref{0608starstar} ensures that there exists $\mu_1<1$ (depending on $\delta$) such that, for all $t>0$,
\begin{align}
\|u \|_{m+1}< \mu_1 \|U_\ast\|_{m+1}.
\end{align}
Conversely, $g(x)$ is strictly decreasing for $x>\|U_\ast \|_{m+1}$. Thus, if
\begin{align}
\|u_0\|_{m+1}> \|U_\ast\|_{m+1},
\end{align}
using \eqref{0608starstar} again, we conclude that exists $\mu_2>1$ (depending on $\delta$) such that
\begin{align}\label{420}
\|u \|_{m+1}> \mu_2 \|U_\ast\|_{m+1}.
\end{align}
This completes the proof. 
\end{proof}

\begin{remark}\label{Fu0FUstar}
We note that if $F(u_0)<F(U_\ast)$, the case $\|u_0\|_{m+1}=\|U_\ast\|_{m+1}$ cannot occur. Indeed, it follows directly from Proposition \ref{picture}(i) that $F(u_0) \ge F(U_\ast)$ whenever $\|u_0\|_{m+1}=\|U_\ast\|_{m+1}=P_\ast$. 
\end{remark}

\subsection{Finite-time blow-up} \label{sub42}

In this subsection, we prove a finite-time blow-up result for \eqref{tfsystem} with large initial data. We employ the standard argument relying on the evolution of the second moment of solutions, following the approach developed in \cite{JL92}.

Let $u(x,t)$ be a weak solution to \eqref{tfsystem} with the maximal existence time $T_{max}$. We define the blow-up criterion
\begin{align}\label{blowcri}
&\|\nabla u(\cdot,t)\|_2 +\|u(\cdot,t)\|_{m+1}<\infty,\quad \text{for } 0<t<T_{max}, \nonumber\\
&\displaystyle \limsup_{ t \to T_{max}} \left(\|\nabla u(\cdot,t)\|_2 +\|u(\cdot,t)\|_{m+1}\right)=\infty.
\end{align}
We first introduce the second moment by
\begin{align}
m_2(t):=\int_{\R^d} |x|^2 u(x,t) dx,
\end{align}
and obtain the following identity.
\begin{lemma}
Let $u$ be a weak solution to \eqref{tfsystem} on $[0,T)$ for some $T \in (0,\infty]$. Then the second moment $m_2(t)$ satisfies
\begin{align}\label{m22t}
\frac{d}{dt} m_2(t)=2(d+2) F(u)-2 \frac{dm-(d+2)}{m+1} \int_{\R^d} u^{m+1} dx,\quad t \in [0,T].
\end{align}
\end{lemma}
\begin{proof}
Integrating by parts in \eqref{tfsystem}, we compute
\begin{align}
\frac{d}{dt} m_2(t) &= 2 \int_{\R^d} x \cdot (u \nabla (\Delta u)) dx+\frac{2m}{m+1} \int_{\R^d} x \cdot \nabla u^{m+1} dx \nonumber \\
&=-2d \int_{\R^d} u \Delta u dx-2\int_{\R^d} (x \cdot \nabla u) \Delta u dx-\frac{2dm}{m+1} \int_{\R^d} u^{m+1} dx \nonumber \\
&=2d \int_{\R^d} |\nabla u|^2 dx-(d-2) \int_{\R^d} |\nabla u|^2 dx  -\frac{2dm}{m+1} \int_{\R^d} u^{m+1} dx \nonumber \\
&=(d+2) \int_{\R^d} |\nabla u|^2 dx-\frac{2dm}{m+1} \int_{\R^d} u^{m+1} dx \nonumber \\
&=2(d+2) F(u)-2 \frac{dm-(d+2)}{m+1} \int_{\R^d} u^{m+1} dx,
\end{align}
where we have used
\begin{align*}
2\int_{\R^d} (x \cdot \nabla u) \Delta u dx &=-2 \int_{\R^d} \nabla(x \cdot \nabla u) \cdot \nabla u dx \\
&=-2 \int_{\R^d} |\nabla u|^2 dx-2 \int_{\R^d} x \cdot D^2 u \nabla u dx \\
&=-2 \int_{\R^d} |\nabla u|^2 dx-\int_{\R^d} x \cdot \nabla (|\nabla u|^2)dx \\
&=(d-2) \int_{\R^d} |\nabla u|^2 dx.
\end{align*}
This completes the proof.
\end{proof}

We now state our finite-time blow-up result.
\begin{theorem}(Finite-time blow-up)\label{finiteblowup}
Let $1+2/d<m<(d+2)/(d-2)$. Let $u(x,t)$ be a weak solution defined in Definition \ref{weakdefine} on $[0,T_w)$. Under assumptions \eqref{initial}, \eqref{Fu0} and \eqref{dayum1}, then $T_w<\infty$ and $u(x,t)$ blows up in finite time, in the sense that $\displaystyle \limsup_{t \to T_w}\|u(\cdot,t)\|_{m+1}=\infty$.
\end{theorem}
\begin{proof}
%Here we show the formal computations, the passing to the limit from the approximated problem \eqref{kseps} can be done according to \cite[Theorem 2.11]{BL13} without any further complications. 

The key tool is the time evolution of the second moment. Recalling that $m>1+2/d$, we use \eqref{m22t} together with \eqref{Fu0} and estimate \eqref{dayum} from Proposition \ref{belowabove} (with some $\mu_2>1$) to obtain
\begin{align}\label{m2t}
\frac{d}{dt} m_2(t) & < 2(d+2) F(u_0)-2 \mu_2 \frac{dm-(d+2)}{m+1} \int_{\R^d} U_\ast^{m+1} dx \nonumber \\
&<2(d+2) F(U_\ast)-2 \mu_2 \frac{dm-(d+2)}{m+1} \int_{\R^d} U_\ast^{m+1} dx \nonumber \\
&=2(1-\mu_2)\frac{dm-(d+2)}{m+1} \|U_\ast\|_{m+1}^{m+1}<0.
\end{align}
Here we have used the fact that $F(u)$ is non-increasing in time and Corollary \ref{Fus}. It then follows from \eqref{m2t} that there exists $T>0$ such that $\displaystyle \lim_{t \to T} m_2(t)=0$.

Next, by H\"{o}lder's inequality we have
\begin{align}
\int_{\R^d} u(x) dx =& \int_{|x| \le R} u(x) dx+\int_{|x|>R} u(x) dx \nonumber \\
\le & C R^{\frac{dm}{m+1}} \|u\|_{m+1} +\frac{1}{R^2} m_2(t).
\end{align}
Choosing $R=\left( \frac{C m_2(t)}{\|u\|_{m+1}}  \right)^{\frac{m+1}{dm+2(m+1)}}$ results in
\begin{align}
M \le c \|u\|_{m+1}^{\frac{2(m+1)}{dm+2(m+1)}} m_2(t)^{\frac{dm}{dm+2(m+1)}},
\end{align}
which implies
\begin{align}
\displaystyle \limsup_{t \to T} \|u(\cdot,t)\|_{m+1} \ge \displaystyle \lim_{t \to T} M^{\frac{dm+2(m+1)}{2(m+1)}} m_2(t)^{-\frac{dm}{2(m+1)}}=\infty.
\end{align}
Since $1<m+1<\frac{2d}{d-2}$, we further conclude that
\begin{align}
\displaystyle \limsup_{t \to T} \|\nabla u\|_2=\infty
\end{align}
by virtue of the GNS inequality \eqref{GNS1}. Thus the proof is completed. 
\end{proof}

\subsection{Global existence} \label{sub41}

In this subsection, we present a global existence result for \eqref{tfsystem} with small initial data.

\begin{theorem}(Global existence)\label{globalexistence}
Let $1+2/d<m<(d+2)/(d-2)$. Assume that \eqref{initial}, \eqref{Fu0} and \eqref{xiaoyum1} hold. Then there exists a weak solution to \eqref{tfsystem} such that for all $0<t<\infty$,
\begin{align}
\|u(\cdot,t)\|_{m+1}+\|\nabla u(\cdot,t)\|_2 \le C\left(\|u_0\|_1,\|\nabla u_0\|_{2} \right)
\end{align}
and 
\begin{align}
\int_0^t \int_{\R^d} u \left|\nabla(\Delta u + u^m)\right|^2 dxds \le C\left(\|u_0\|_1,\|\nabla u_0\|_{2} \right).
\end{align}
Moreover, we have
\begin{align}\label{m2blow}
\displaystyle \lim_{t \to \infty} m_2(t)=\infty.
\end{align}
\end{theorem}
\begin{proof}
The proof can be divided into two parts. Step 1 establishes the global existence of weak solutions. Step 2 shows the blow-up of the second moment and thereby proves \eqref{m2blow}.

{\it\textbf{Step 1}} (Global existence) \quad Under assumption \eqref{xiaoyum1}, it follows from \eqref{xiaoyum} in Proposition \ref{belowabove} that for any $t>0$, there exists $\mu_1<1$ such that
\begin{align}\label{260109}
\|u(\cdot,t)\|_{m+1}<\mu_1 \|U_\ast\|_{m+1}.
\end{align}
Furthermore, for any $t>0$, the non-increasing of the free energy $F(u)$ in time implies that
\begin{align}\label{260110}
\frac{1}{2} \|\nabla u(\cdot,t)\|_2^2  \le \frac{1}{m+1} \|u(\cdot,t)\|_{m+1}^{m+1}+ F(u_0) \le C\left(\|U_\ast\|_{m+1},F(u_0)\right),
\end{align}
and 
\begin{align}
\int_0^t \int_{\R^d} u \left|\nabla(\Delta u + u^m)\right|^2 dxds \le F(u_0)+\frac{1}{m+1} \|u(\cdot,t)\|_{m+1}^{m+1} \le C\left(\|U_\ast\|_{m+1},F(u_0)\right).
\end{align}
By combining the Sobolev embedding theorem with H\"{o}lder's inequality for $1<r<\frac{2d}{d-2}$, we infer from \eqref{260110} that
\begin{align}
\| u\|_{r} \le C, \text{ for any } 1 \le r \le \frac{2d}{d-2}.
\end{align}
Therefore, for any $T>0$, the following uniform regularity estimates hold:
\begin{align}
&u \in L^\infty(0,T;L^1 \cap H^1 (\R^d)), \label{11} \\
&u \in L^\infty(0,T;L^{r}(\R^d)),\text{ for any } 1 \le r \le 2d/(d-2), \label{22} \\
&u^{1/2} \nabla \Delta u \in L^2(0,T;L^2(\R^d)), \nonumber \\
&u^{1/2} \nabla u^m \in L^2(0,T;L^2(\R^d)). \nonumber
\end{align}
On the other hand, the second moment satisfies
\begin{align}\label{m2bdd}
m_2(t) \le m_2(0)+2(d+2) F(u_0)t, \quad \text{for any } 0<t<T. 
\end{align}
Thus far, uniform bounds on the competing terms in the free energy have been established, which serve as the key ingredient to exploit the gradient-flow structure of \eqref{tfsystem}. Following the argument framework in Section 4 of \cite{Jose24}, for $1+2/d<m<(d+2)/(d-2)$, we can establish the existence of a weak solution to \eqref{tfsystem} on $[0,T]$ for any $0<T<\infty$. The proof relies on the JKO scheme developed therein. Since our estimates yield exactly the required assumptions for that approach, the remaining steps including discrete approximation energy estimates and passage to the limit follow verbatim. We thus omit the repetitive details.

{\it\textbf{Step 2}} (Proof of \eqref{m2blow}) \quad  Suppose, for contradiction, that there exists an increasing sequence of times $\{t_k\}_{k \in \mathbb{N}} \to \infty$ such that 
\begin{align}
m_2(t_k)=\int_{\R^d} |x|^2 u(x,t_k) dx
\end{align}
is bounded. Since $u \in L^1(\R^d) \cap H^1(\R^d)$, we first claim that there exist a subsequence (still denoted $t_k$) and a function $u_\infty$ such that as $k \to \infty$,
\begin{align}
& u(\cdot,t_k) \rightharpoonup u_\infty \text{ in } H^1(\R^d), \nonumber \\
& u(\cdot,t_k) \rightarrow u_\infty \text{ in } L^1(\R^d), \label{L1} \\
& u(\cdot,t_k) \rightarrow u_\infty \text{ in } L^{m+1}(\R^d) \label{Lm1} 
\end{align}
and $\|u_\infty\|_1=M$. Moreover, the second moment of the limit satisfies
\begin{align}
0<\int_{\R^d} |x|^2 u_\infty dx \le \displaystyle \liminf_{k \to \infty}\int_{\R^d} |x|^2 u(\cdot, t_k) dx<\infty,
\end{align}
since concentration toward a Dirac delta at the origin is ruled out by the uniform bound \eqref{11}. 

On the other hand, a direct consequence of \eqref{260109} is that the free energy $F(u)$ is bounded from below:
\begin{align}
F(u_0)-\displaystyle \liminf_{k \to \infty} F(u)(t_k) =\lim_{k \to \infty} \int_0^{t_k} \left( \int_{\R^d} u(x,s) \left| \nabla \left(\Delta u+u^m  \right) \right|^2 dx \right) ds.
\end{align}
As a consequence, the Fisher information is integrable and 
\begin{align}
\lim_{k \to \infty} \int_{t_k}^\infty \left( \int_{\R^d} u(x,s) \left| \nabla \left(\Delta u+u^m  \right) \right|^2 dx \right) ds=0,
\end{align}
which implies that, up to the extraction of a subsequence, the limit $u_\infty$ as $k \to \infty$ satisfies
\begin{align}\label{260123}
\nabla(\Delta u_\infty+u_\infty^m)=0 \quad \text{a.e. in supp}(u_\infty).
\end{align}
This follows by arguments similar to those in the existence proof in Section 2 of \cite{BCM08}. Therefore, $u_\infty$ is a radially symmetric steady-state solution to \eqref{second} with mass $M$. By Proposition \ref{picture}(ii), we have
\begin{align}\label{Usinfty}
F(U_\ast) \le F(u_\infty).
\end{align}
However, \eqref{Lm1} yields 
\begin{align}
F(u_\infty) & =\frac{1}{2} \|\nabla u_\infty \|_{2}^2-\frac{1}{m+1} \|u_\infty\|_{m+1}^{m+1} \nonumber \\
& \le \displaystyle \liminf_{k \to \infty} \frac{1}{2} \|\nabla u(\cdot, t_k)\|_{2}^2-\frac{1}{m+1} \|u(\cdot,t_k)\|_{m+1}^{m+1} \nonumber \\
& = \displaystyle \liminf_{k \to \infty} F(u(\cdot,t_k)) \le F(u_0)<F(U_\ast),
\end{align}
which contradicts \eqref{Usinfty}. Hence, \eqref{m2blow} holds. 

To complete the proof of \eqref{m2blow}, it remains to justify the convergences \eqref{L1} and \eqref{Lm1} used above. We first establish \eqref{L1}. Let $0 \le \chi_R \le 1$ be a function in $C_0^\infty(\R^d)$ such that $\chi_R=1$ on $B_R(0)$ and $\chi_R=0$ on $\R^d\setminus B_{2R}(0)$. We decompose
\begin{align}
M=\int_{\R^d} u(\cdot,t_k) dx=\int_{\R^d} u(\cdot,t_k) \chi_R dx+\int_{\R^d} u(\cdot,t_k) (1-\chi_R) dx.
\end{align}
Then for sufficiently large $R$,
\begin{align}\label{largeR}
\int_{\R^d} u(\cdot,t_k) (1-\chi_R) dx \le \int_{|x| \ge R} u(\cdot, t_k) dx \le \frac{1}{R^2} \int_{|x| \ge R} |x|^2 u(\cdot,t_k) dx<\varepsilon,
\end{align}
while
\begin{align*}
\displaystyle \lim_{k \to \infty} \int_{\R^d} u(\cdot,t_k) \chi_R dx=\int_{\R^d} u_\infty \chi_R dx.
\end{align*}
Hence, we have
\begin{align}
\left| M-\displaystyle \lim_{k \to \infty} \int_{\R^d} u(\cdot,t_k) \chi_R dx \right|<\varepsilon, \quad \text{as } R \to \infty,
\end{align}
which implies
\begin{align}\label{massconservation}
\int_{\R^d} u_\infty dx=M>0.
\end{align}
This establishes \eqref{L1}. 

We next prove \eqref{Lm1}. Using the condition $1<m+1<2d/(d-2)$ and the compact Sobolev embedding $W^{1,2}(B_R(0)) \hookrightarrow L^q(B_R(0))$ for $1 \le q<2d/(d-2)$, we obtain local strong convergence in $L^{m+1}(B_R(0))$ for any fixed $0<R<\infty$. Moreover, it follows from the GNS inequality \eqref{GNS1} that
\begin{align}
\int_{|x| \ge R} u^{m+1} dx & \le C_* \left( \int_{|x| \ge R} u dx  \right)^{\frac{d+2-(d-2)m}{d+2}} \|\nabla u\|_{2}^{\frac{2dm}{d+2}} \nonumber \\
& \le C\left( M,C_*,F(u_0) \right)~ \left( \frac{m_2(t)}{R^2} \right)^{\frac{d+2-(d-2)m}{d+2}}
\end{align}
upon using \eqref{largeR}. Therefore, as $R \to \infty$ and $k \to \infty$,
\begin{align*}
& \int_{\R^d} \left|u(\cdot,t_k)-u_\infty \right|^{m+1} dx=\int_{|x| < R} \left|u(\cdot,t_k)-u_\infty \right|^{m+1} dx+\int_{|x| \ge R} \left|u(\cdot,t_k)-u_\infty \right|^{m+1} dx \\
\le & \int_{|x| < R} \left|u(\cdot,t_k)-u_\infty \right|^{m+1} dx+C \int_{|x| \ge R} \left(\left| u(\cdot,t_k) \right|^{m+1}+\left| u_\infty \right|^{m+1}\right) dx \to 0.
\end{align*}
This completes the proof of Theorem \ref{globalexistence}. 
\end{proof}

\iffalse
\begin{remark}
The variational approach used above fails for $m=\frac{d+2}{d-2}$, since the condition $m+1<\frac{2d}{d-2}$ essential for the implementation of 
the JKP scheme in \cite{Jose24} is no longer satisfied. This indicates that the existence problem for $m=\frac{d+2}{d-2}$ is fundamentally different and, to our knowledge, remains open. 
\end{remark}
\fi

\subsection{Possible dynamic behaviors for $F(u_0)>F(U_\ast)$ }\label{FU0great}

For initial data satisfying $F(u_0) > F(U_\ast)$, energy dissipation no longer prevents the solution from approaching the steady state $U_\ast$. The long-time behavior of global solutions is conjectured to depend on the boundedness of two key quantities:
\begin{align*}
&\textbf{(P1) } \sup_{t \ge 0} \| u(\cdot,t)\|_{m+1} < \infty, \quad \text{(concentration control)} \\
&\textbf{(P2) } \sup_{t \ge 0} \int_{\mathbb{R}^d} |x|^2 u(x,t)\,dx < \infty. \quad \text{(mass confinement)}
\end{align*}
Based on the gradient-flow structure and analogy with related models, we expect the following scenarios:
\begin{itemize}
    \item \textbf{(P1) and (P2):} The solution is expected to converge (possibly up to translation) to $U_\ast$.    
    \item \textbf{(P1) and not (P2):} Mass is expected to escape to infinity while the profile remains bounded in $L^{m+1}$. After a suitable time-dependent translation that follows the escaping core, the profile may converge locally to a steady state.    
    \item \textbf{not (P1) and (P2):} Unbounded concentration growth within a confined region suggests the possibility of infinite-time singularity formation as a Dirac mass, and the free energy would tend to $-\infty$.    
    \item \textbf{not (P1) and not (P2):} The solution may exhibit complex dynamics involving simultaneous concentration growth and mass dispersion, possibly leading to intricate asymptotic patterns.
\end{itemize}

\begin{remark}
With the exception of the first case, the precise asymptotic behavior in the other scenarios depends on further details of the model and the initial data. A rigorous analysis of these cases remains an open challenge.
\end{remark}

\section*{Data Availability Statement}
All data generated or analysed during this study are included in this published article (and its supplementary information files). All data that support the findings of this study are included within the article (and any supplementary files).

\section*{Statements and Declarations}
This research did not receive any specific grant from funding agencies in the public, commercial, or not-for-profit sectors. The author has no competing interests to declare that are relevant to the content of this article.

\end{document}